\documentclass[reqno,a4paper,11pt]{amsart}
\usepackage{color}
\usepackage{amsmath,amssymb,amsfonts,amsthm,enumerate}
\usepackage{hyperref,mathrsfs,accents}
\usepackage[T1]{fontenc}
\usepackage[active]{srcltx}
\usepackage{epsfig}
\usepackage{graphicx}
\usepackage{cite}
\usepackage{tikz}
\usepackage{pgfplots}
\usepackage{tikz-3dplot}
\usepackage{subcaption}
\usepackage{pgfplots}
\usepackage{adjustbox}
\usepackage{changepage}
\usetikzlibrary{patterns,backgrounds}

\tdplotsetmaincoords{60}{115}
\pgfplotsset{compat=newest}

\catcode`@=11 \@addtoreset{equation}{section} \catcode`@=12

\newtheorem{theorem}{Theorem}[section]
\newtheorem{lemma}[theorem]{Lemma}

\newtheorem{remark}[theorem]{Remark}
\newtheorem{proposition}[theorem]{Proposition}

\newcommand{\He}{\mathbb{H}}

\newcommand{\N}{\mathbb{N}}
\newcommand{\R}{\mathbb{R}}
\newcommand{\W}{\mathbb{W}}

\newcommand{\de}{\partial}

\newcommand{\spt}{\mathop{\mathrm{spt}}}

\newcommand{\difsim}{\Delta}

\newcommand{\cA}{\mathcal{A}}
\newcommand{\cB}{\mathcal{B}}

\newcommand{\cE}{\mathcal{E}}

\newcommand{\cL}{\mathcal{L}}

\newcommand{\cQ}{\mathcal{Q}}

\renewcommand{\cB}{\mathscr B}
\renewcommand{\cA}{\mathscr A}
\renewcommand{\cL}{{\mathscr L}^3}

\renewcommand{\Subset}{\subset\!\subset}

\newcommand{\Lip}{\mathrm{Lip}}

\newcommand{\p}{P}

\renewcommand{\phi}{\varphi}
\renewcommand{\rho}{\varrho}
\renewcommand{\epsilon}{\varepsilon}
\renewcommand{\He}{\mathbb{H}^{1}}

\AtBeginDocument{
	\let\div\relax
	\DeclareMathOperator{\div}{div}
}

\author{Roberto Monti}
\address{Dipartimento di Matematica, via Trieste 63, IT-35131 Padova (Italy)}
\email{monti@math.unipd.it}

\author{Giacomo Vianello}
\address{UTIA -- Czech Academy of Sciences, Pod Vodárenskou věží 1143/4, 182 00 Prague 8 (Czech Republic)}
\email{vianello@utia.cas.cz}



\title[Plateau's Problem in $\He$]{Plateau's Problem for intrinsic graphs in the Heisenberg Group}

\thanks{Roberto Monti has been supported by PRIN 2022F4F2LH ``Regularity problems in sub Riemannian structures''. Giacomo Vianello has been supported by: PRIN 2022F4F2LH ``Regularity problems in sub Riemannian structures''; GNAMPA (INdAM) Project 2025 ``Structure of sub-Riemannian hypersurfaces in Heisenberg groups'', CUP ES324001950001.}

\begin{document}
	
	\begin{abstract}
		Using a geometric construction, we solve Plateau's Problem in the Heisenberg group $\He$
		for intrinsic graphs defined on a convex domain $D$, under a smallness condition either on
		the boundary $\partial D$ or on the boundary datum $\phi \in \Lip(\partial D)$.  
		The proof relies on a calibration argument.  
		We then apply these techniques to establish a new regularity result for $H$-perimeter minimizers.
	\end{abstract}

	\maketitle
	\large
	\section{Introduction}
	The classical Plateau's Problem—minimizing the area of a surface spanning a given boundary curve—can be formulated in several different ways: for graphs, parametric surfaces, Caccioppoli sets, currents, and others. The precise formulation and its solution depend 
	on the chosen notion of surface and area, as well as on how the boundary condition is imposed.
	A nice overview can be found in \cite{Stein_book_2014,Giusti_book}.
	
	In this paper, we study the Plateau's Problem in the first Heisenberg group $\He$, where the notion of area is defined via the horizontal perimeter of a set in the sense of De Giorgi, along with the associated area formulas.
	In this context, the existence of solutions to the problem was first investigated by Garofalo and Nhieu in \cite{Garofalo_Nhieu_1996} for sets with finite perimeter. For the case of $t$-graphs in $\mathbb{H}^n$, the existence was first studied by S. D. Pauls in \cite{SDPauls_Min_Surf_2004}, adapting techniques related to the Bounded Slope Condition, which are indeed applicable to general area-type functionals, and later by other authors in \cite{PSCTV_2015,Cheng_Hwang_Yang_2009,GPPV_2024,Pozuelo_Verz_2024,DLPT_2022}.

	Our focus is on \emph{intrinsic graphs}—that is, graphs defined along the integral curves of a left-invariant vector field. This notion of graph appears to be the most appropriate, as the boundary of sets with finite $H$-perimeter is rectifiable precisely in this sense: up to a negligible set, it can be covered by a countable union of $C^1_H$-regular  \emph{intrinsic graphs}, see \cite[Theorem 7.1]{FSSC_Rect_Per_M_Ann_2001}. In fact, our construction yields solutions to the Plateau's  Problem that are intrinsic graphs.
	
	In this setting, the question of the existence of solutions becomes more challenging. On the one hand,
	when passing from 
	$t$-graphs to intrinsic graphs, the area functional is strongly non-linear and non-convex.
	As a consequence, stationary points of the area may be unstable \cite{DGN_Stat_Unst_2008}.
	On the other hand, there is, so far, no satisfactory regularity theory for 
	$H$-minimal surfaces
	(i.e., stationary points of the horizontal perimeter). Indeed, several examples of 
	$H$-minimal surfaces with low regularity are known
	\cite[Theorem C]{SDPauls_Low_Reg_2006}.

	However, when $n=1$, the situation is more favorable: smooth $2$-dimensional $H$-minimal surfaces in $\He$ without characteristic points are ruled (i.e., foliated by horizontal segments) and therefore exhibit a certain rigidity (see for instance \cite{Galli_Ritore_2015,Giovannardi_Ritore_2024,CHM_Min_surf_Pseudoherm_05}).
	This observation is the starting point of our study.
	
	The first Heisenberg group $\He$ is the 3-dimensional Euclidean space $\R^3$ equipped with the group operation:
	\begin{equation*}
		(x,y,t) \cdot (\xi,\eta,\tau) = \big(x + \xi,\ y + \eta,\ t + \tau + 2(y\xi - x\eta)\big),\qquad 
		(x,y,t) ,  (\xi,\eta,\tau) \in\He.
	\end{equation*}
	The Lie algebra of left-invariant vector fields is spanned by the three vector fields defined, at each point $p = (x,y,t) \in \He$,  as:
	\begin{equation*}
		X(p) := \dfrac{\partial}{\partial x} + 2y \dfrac{\partial}{\partial t}, \qquad
		Y(p) := \dfrac{\partial}{\partial y} - 2x \dfrac{\partial}{\partial t}, \qquad
		T := \dfrac{\partial}{\partial t}.
	\end{equation*}
	We denote by $H$ the \emph{horizontal sub-bundle} of $\He$, that is,
	\begin{equation*}
		H_p := \mathrm{span} \{ X(p), Y(p) \} , \qquad \text{for $p \in \He$}.
	\end{equation*} 
	
	A $C^1$ vector field $V$ defined in an open set $\Omega \subset \He$ is said to be horizontal if $V(p) \in H_p$ for every $p \in \Omega$. In this case, we write $V \in C^1(\Omega; H)$.
	If  $V\in C^1_c(\Omega;H)$, i.e.,  $V =v_1 X + v_2  Y$ with $v_1,v_2\in C_c^1(\Omega)$, we let $\| V\| _\infty = \max _ {p\in\Omega} (v_1(p)^2+v_2(p)^2)^{1/2}$. Notice that
	\[
	\div V := Xv_1 + Yv_2 
	\]
	is the standard divergence of $V$.
	
	The Haar measure of $\He$ is the Lebesgue measure $\cL$.
	We say that a $\cL$-measurable set $E \subset \He$ has \emph{finite H-perimeter} in $\Omega$ if
	\begin{equation*}
		\p_H(E;\Omega) := \sup \left \{ \int_E \div V d \cL \, : \, V \in C^1_c(\Omega;H) , \, \|V\| _\infty \leq 1 \right \} < \infty .
	\end{equation*}
	The quantity $\p_H(E;\Omega)$ is called the \emph{H-perimeter} of $E$ in $\Omega$, see \cite{FSSC_Rect_Per_M_Ann_2001}.
	When $\partial E \cap \Omega$ is the graph of a function with suitable regularity, the $H$-perimeter can be represented by   area-type formulas.
	We are interested in the case of intrinsic graphs, and, without loss of generality, we focus on graphs defined along the vector field $X$.
	
	Let $\Phi :\R\times \He\to \He$ be the flow of $X$, and define $\Phi_s : \He\to\He$ by  $\Phi_s (p)=\Phi(s,p)$, for $s\in\R$ and $p\in\He$.
	We denote by $\mathbb{W} :=  \{ p = (x,y,t) \in \He \, : \, x = 0 \}$ the vertical plane (subgroup) ‘‘orthogonal to $X$'', and we fix a domain 
	$D \subset \mathbb{W}$. The intrinsic graph (along $X$) of a function $u: D \to \R$ is defined as
	\begin{equation*}
		S_u  =\big \{ \Phi_{s}(y,t) = (s, y,  t + 2  y s)\in\He  \, : \, (y,t) \in D,\, s= u(y,t)\big  \}.
	\end{equation*}
	Similarly, the intrinsic epigraph of $u$ is \[ 
	E_u  =\big \{ \Phi_{s}(y,t) \in\He  \, : \, (y,t) \in D,\, s> u(y,t)\big  \}.
	\]
	The (left) cylinder over $D$ is denoted by
	\[
	D\cdot \mathbb V  = \big \{ \Phi_{s}(y,t) \in\He  \, : \, (y,t) \in D,\, s\in\R\big  \},
	\]
	where $\mathbb V = \big \{(x,0,0)\in\He\, :\, x\in\R\big\} $ is the complementary subgroup of $\mathbb W$.
	
	If $D\subset \mathbb W $ is bounded and $u\in\Lip(D)$, then $P_H(E_u; D\cdot \mathbb V   )<\infty$ and there holds the area formula, see \cite{FSSC_AFgr_Subm_2007}, 
	\begin{equation} \label{eq:hor_per_X_intr_epi}
		\cA_D(u) := \p_H(E_u; D\cdot \mathbb V  ) =  \int_{D} \sqrt{1 + \cB u(y,t) ^2} \, dy dt <\infty ,
	\end{equation}
	where
	\[
	\cB u  = \Big( \dfrac{\de }{\de y} - 4 u  \dfrac{\de }{\de t}\Big) u  = \dfrac{\de u}{\de y} - 2 \dfrac{\de }{\de t}u^2.
	\]
	Notice that the area functional $u\mapsto \cA_D(u)$ is not convex.
	
	The Burgers' operator $u\mapsto \cB u$   is    well-defined  also  in distributional sense,
	and formula \eqref{eq:hor_per_X_intr_epi} remains valid
	under the weaker assumption that $u$ is a continuous function
	such that $\cB u \in C(D)$, see \cite[Proposition 2.22]{ASCV_Intr_reg_hyp_2006}, or even in the weaker situation of $\cB u \in L^\infty(D)$, see \cite[Theorem 1.6]{CMP_Approx_Intr_Lip_2014}.
	We say that $u\in C(D)$ is \emph{intrinsic Lipschitz} if
	$\cB u \in L^\infty(D)$, and that $u\in C^{1,h}_H (D)$ (resp. $u\in C^{1,h}_{H,loc} (D)$) if $\cB u \in C^{0,h}(D)$ (resp. $\cB u \in C^{0,h}_{loc}(D)$).

	In this paper, we address the following question:  
	\vspace{0.2cm}
	\begin{equation} \tag{PP}\label{PP}
		\begin{minipage}{0.8\linewidth}
			\textit{Given a bounded 
				open set $D \subset \mathbb{W}$ and a function $\phi : \de D \to \R$, 
				can one find a function $u\in C(\bar  D )$ such that $u|_{\de D} = \phi$ and $\cA_D(u)$ is minimal among
				a suitable class of competitors satisfying the same boundary condition as $u$? }
		\end{minipage} 
	\end{equation}
	\vspace{0.2cm}

\noindent 
We will provide an affirmative answer to this version of the Plateau's Problem
for a large class of convex domains $D$ and for Lipschitz boundary data $\phi : \de D \to \R$, also proving some regularity for the minimizer.
The issue about the ‘‘suitable class of competitors'' is delicate and will be clarified  later.
We are able to prove existence under a \emph{smallness} condition that jointly involves 
$\de D$ and the boundary datum $\phi$. 

\newcommand{\strict}
{\marginpar{STRICTLY} }

Let us fix the interval  $I = (0,\bar{t})\subset \R$, for some $\bar{t} > 0$.
We consider domains $D \subset \mathbb{W}$ of the form
\begin{equation}\label{D}
	D = D_{\gamma_1,\gamma_2} := \{ (y,t) \in \mathbb{W} \, : \, t \in I, \, \gamma_1(t) < y < \gamma_2(t) \} ,
\end{equation}
where $\gamma_1, \gamma_2 \in\Lip(I)$ are such that
\begin{equation}\label{gamma}
	\gamma_1(t) < 0 < \gamma_2(t) \quad \text{for all $t \in I$} \qquad \text{and} \qquad \gamma_1(0) = \gamma_1(\bar{t}) = \gamma_2(0) = \gamma_2(\bar{t}) = 0 .
\end{equation}
The domain $D=D_{\gamma_1,\gamma_2}$ is convex  if and only if $\gamma_1$ is   convex and $\gamma_2$ is   concave in $I$. 
Notice that $D$ is not allowed to be a disk, because in this case $\gamma_1,\gamma_2\notin\Lip(I)$.

Given a continuous function $\phi : \de D \to \R$, we define the compositions  
\begin{equation}\label{phi}
	\phi_1(s) := \phi(\gamma_1(s),s) \qquad \text{and} \qquad \phi_2(s) := \phi(\gamma_2(s),s),\qquad s\in I .
\end{equation}
We   also consider the  mappings
\begin{align} \label{eq:def_p_1}
	& \bar{I} \ni s \mapsto p_1(s) := \Phi_{\phi_1(s)}(\gamma_1(s),s) 
	= \big (\phi_1(s),  \gamma_1(s), s + 2 \gamma_1(s) \phi_1(s)\big ) , 
	\\
	& \bar{I} \ni s \mapsto p_2(s) := \Phi_{\phi_2(s)}(\gamma_2(s),s) 
	= \big (\phi_2(s),  \gamma_2(s),  s + 2 \gamma_2(s) \phi_2(s)\big ) . \label{eq:def_p_2}
\end{align}
The images of $p_1$ and $p_2$ form the intrinsic graph of $\phi$, i.e.,  $p_1(\bar I)\cup p_2(\bar I)= S_\phi$.
With a slight abuse of notation, we adopt the following conventions:
\begin{align*}
	& \| \gamma \|_\infty := \max \{ \| \gamma_1 \|_{\infty} , \| \gamma_2 \|_{\infty} \},
	\qquad \, \, \Lip(\gamma) := \max\{ \Lip(\gamma_1), \Lip(\gamma_2) \} , \\
	& \| \phi \|_\infty  := \max \{ \| \phi_1 \|_{\infty} , \| \phi_2 \|_{\infty} \} ,
	\qquad \Lip(\phi) := \max \{ \Lip(\phi_1), \Lip(\phi_2) \} ,
\end{align*}
and  we denote by $\zeta=\zeta_{\gamma_1,\gamma_2,\phi_1,\phi_2}$ the following parameter,
which will accompany us throughout this work:
\begin{equation}\label{zeta}
	\zeta := 4\big (\| \gamma \|_\infty + \Lip(\gamma)\big) \big(\| \phi \|_\infty  + \Lip(\phi)\big) .
\end{equation}

The first step in the construction of the solution to the Plateau Problem (\ref{PP}) 
is the following result:

\begin{theorem} \label{thm:main_1_intro}
	Let $\gamma_1,\gamma_2\in\Lip(I)$ be as in \eqref{gamma}, $D = D_{\gamma_1,\gamma_2} \subset \mathbb{W}$ as in \eqref{D},
	$\phi \in C(\de D)$ be with $\phi_1,\phi_2\in\Lip(I)$ as in \eqref{phi}, and let $\zeta$ be the parameter   in \eqref{zeta}.
	Then we have:  
	\begin{itemize} 
		\item[(i)] If $\zeta < 1$,    there exists   $\rho\in \Lip( Q ; \R^3)$,  $Q := I \times (0,1)$,  such that:
		\begin{align}
			\rho(\de Q) &= p_1(\bar{I}) \cup p_2(\bar{I})=S_\phi;
			\\
			\rho(h,s) & = (1-h) \rho(0,s) + h \rho(1,s), \qquad \text{\quad \text{for $(h,s) \in Q$;}}
			\\
			\frac{\de}{\de h} 		\rho(h,s) & =	 \rho(1,s) - \rho(0,s) \in H_{\rho(h,s)}, \qquad\,\,\,\,   \text{for $(h,s) \in Q$.}  
		\end{align}
		\noindent
		In particular, the parametric surface $R _\phi := \rho(Q)$ is foliated 
		by straight horizontal segments spanning $S_\phi$.
		\medskip
		
		\item[(ii)] If $\zeta < (\sqrt{129} - 11)/4$ and $D$ is convex, there exists   $u \in \Lip(D) \cap C^{1,1}_H(D)$ such that:
		\begin{equation}
			u = \phi \quad \text{on $\de D$} \qquad \text{and} \qquad \text{$S_u=R_\phi $.}
		\end{equation}
		In addition, if $\gamma_1, \gamma_2, \phi_1, \phi_2 \in C^k(I)$, for some $k \geq 1$, then $u \in C^k(D)$.
		\medskip
	\end{itemize}
\end{theorem}

The surface $R_\phi$---the horizontally ruled surface spanning $S_{\phi}$---is the candidate solution to the Plateau's Problem (\ref{PP}). 
The fact that it is a minimizer of the area functional $\cA_D$ in a suitable class of competitors will be proved by a calibration argument.

While $R_\phi=S_u$ is a left intrinsic graph, the calibration is constructed using the right-invariant structure of $\He$.
We   consider the flow $\Phi^r:\R\times\He\to\He$ of the right-invariant vector field  $X^r$ defined, at each point $p = (x,y,t) \in \He$, as:
\[
X^{r} (p)= \dfrac{\partial}{\partial x} - 2y \dfrac{\partial}{\partial t}.
\]
We let $\Phi^r_s(p) =\Phi^r(s,p)$ for $s\in\R$ and $p\in\He$, and  we consider the right cylinder over 
a domain $D^r\subset\mathbb W$, 
\begin{equation*} 
	\mathbb V \cdot D^ r:= \{ \Phi^r_s(y,t)=(s, y,  t - 2  y s) \in\He  \, : \, (y,t)  \in D^r, \, s \in \R \} .
\end{equation*}
The right graph and epigraph of a function $v: D^r\to \R$ are defined as:
\begin{equation*}
	\begin{split}
		S^r_v  & =\big \{ \Phi_{s}^r(y,t) \in\He  \, : \, (y,t) \in D^r,\, s= v(y,t)\big  \},
		\\
		E^r_v  &=\big \{ \Phi_{s}^r(y,t) \in\He  \, : \, (y,t) \in D^r,\, s> v(y,t)\big  \}.
	\end{split}
\end{equation*}
\noindent
The calibration can be  constructed in the cylinder $\mathbb V \cdot D^ r $ thanks to the following:

\begin{proposition} \label{PROP}
	Let $u \in \Lip(D)$ and $S_u = R_\phi$ be given by (i) and (ii) in Theorem \ref{thm:main_1_intro} for a convex domain $D = D_{\gamma_1,\gamma_2} \subset \mathbb{W}$.
	If $\zeta < (\sqrt{721} - 25)/48$, then there exists an open set $D^r \subset \mathbb{W}$ and $u^r \in \Lip_{loc}(D^r)$ such that $S_u = S_{u^r}^r$.
\end{proposition}

\noindent 
Namely, for small $\zeta$ the left intrinsic graph $S_u$ is also a right intrinsic graph over a suitable domain $D^r\subset\mathbb W$.

Now consider the measure theoretic inner normal $\nu_{E_u}=\omega_1 X + \omega_2 Y$: this is the unit horizontal vector field,  defined for $\mathscr H^2$-a.e.~$p\in R_\phi=S_u$, such that for all $V= v_1 X + v_2 Y \in C_c^1 (D\cdot \mathbb V;H)$ there holds
\begin{equation}
	\label{NN}
	\int_{D\cdot \mathbb V} \div V d\cL = - \int _{S_u} ( \omega_1 v_1 + \omega_2 v_2)  d \mu_{E_u},
\end{equation}
where $ \mu_{E_u}$ is the $H$-perimeter measure associated with the epigraph $E_u$ (see \cite{FSSC_Meyers_Serrin_96}).
The existence of $\nu_{E_u}$ is given by Riesz' representation theorem and its regularity follows from the regularity of $S_u$.

The $0$-homogeneous extension of $\nu_{E_u}$  along the integral lines of $X^r$ defines a unit horizontal vector field $\omega=\omega_1 X + \omega_2 Y$  in the right cylinder $\mathbb V\cdot D^r$ such that $\omega =\nu_{E_u}$ on $S_u$ and 
\begin{equation}\label{omega}
	\div \omega (p) = X\omega_1(p) + Y\omega_2(p) =0, \qquad \textrm{for $\cL$-a.e.~$p\in \mathbb V\cdot D^r$}.
\end{equation}
This follows from the fact that the diffeomorphisms $\Phi^r_s:\He\to\He$, $s\in\R$, are contact and isometric, and from the ruled property of $S_u$.

The vector field $\omega$ calibrates $S_u$ in $\mathbb V\cdot D^r$ in the sense of Barone Adesi-Serra Cassano-Vittone \cite{BaroneAdesi_SerraC_Vitt_2007}, and $u$ solves   Plateau's Problem (\ref{PP}) in the following sense:

\begin{theorem} \label{thm:main_2_intro}
	
	Let $\gamma_1,\gamma_2\in\Lip(I)$ be as in \eqref{gamma}, and let $D = D_{\gamma_1,\gamma_2} \subset \mathbb{W}$ be convex,
	let $\phi \in C(\de D)$ with $\phi_1,\phi_2\in\Lip(I)$ as in \eqref{phi}, and let the parameter $\zeta$ in \eqref{zeta} be such that:
	\begin{equation*}
		\zeta < \dfrac{\sqrt{721} - 25}{48} .
	\end{equation*}
	Finally, let $D^r\subset \mathbb W$ and $u^r$ be as in Proposition \ref{PROP}.
	Then the function $u\in \Lip(D)$ constructed in Theorem \ref{thm:main_1_intro} satisfies the minimality property:
	\begin{equation*}
		\cA_D(u) = \p_H(E_u ; D\cdot \mathbb V)=
		\p_H(E^r_{u^r}  ;\mathbb V\cdot  D ^r)
		\leq 
		\p_H(F;\mathbb V\cdot  D ^r),
	\end{equation*}
	for any $\cL$-measurable set $F\subset \He$ such that $E^r_{u^r} \difsim F \Subset \mathbb V\cdot  D ^r$.
\end{theorem}
\noindent

It is natural to ask which conditions must be satisfied by an intrinsic Lipschitz map $v : D \to \R$ in order to have $\cA_D(u) \leq \cA_D(v)$. In Lemma \ref{lem:minimization_intr_gr}, we provide a sufficient condition on $v$ ensuring this property, namely that $v|_{\de D} = \phi$ and that $S_v$ is the right intrinsic graph of some $v^r : D^r \to \R$.
Despite this restriction on competitors, $u$ minimizes $\cA_D$ in a sufficiently large class that includes compact variations of $u$. This ensures that $u$ is stationary for $\cA_D$ (see Remark \ref{rem:u_stationary}).

In the final section, we show how the theory developed in this work can be used to obtain regularity results for minimizers of $\cA_D$. In Theorem \ref{thm:reg}, we prove that a minimizer is horizontally ruled and belongs to $\Lip(D) \cap C^{1,1}_H(D)$, starting from an initial regularity lower than Lipschitz. This result improves upon Theorem 3.7 in \cite{Giovannardi_Ritore_2024} by Giovannardi and Ritoré,
which concerns the broader setting of stationary points of $\cA_D$ (see also \cite{Galli_Ritore_2015}).
More generally, most regularity results in the literature deal with Lipschitz stationary points of $\cA_D$, see \cite{CCM_H1_2009,CCM_smooth_2010}. The motivation is, again, related to the non-regularity of $H$-minimal surfaces,
which prevents the hope of proving more general regularity results for stationary points of $\cA_D$. An exception is \cite{Young_Harmonic_2023}, where the author addresses the problem of approximating the $H$-area  of intrinsic Lipschitz graphs with small intrinsic gradient by means of a suitable intrinsic Dirichlet energy, with the ultimate goal of reproducing the Euclidean approach to studying regularity.

We conclude by highlighting an interesting consequence for the Bernstein problem in $\He$, which follows from our regularity Theorem \ref{thm:reg} in combination with \cite[Theorem 1.3]{Young_Area_min_ruled_22}. Assume that $u : \mathbb{W} \to \R$ is a function such that its epigraph $E_u$ is $H$-perimeter-minimizing in $\He$ with respect to compact perturbations, that is,
\begin{equation*}
	\p_H(E_u;\Omega') \leq \p_H(F;\Omega') \qquad \text{whenever} \qquad E_u \difsim F \Subset \Omega' \Subset \He.
\end{equation*}
The Bernstein problem asks whether $E_u$ must be a vertical plane. Theorem 1.3 in \cite{Young_Area_min_ruled_22} states that this is the case when $E_u$ is foliated by horizontally straight segments. This allows us to conclude that $E_u$ is a vertical plane whenever $u$ satisfies the assumptions of Theorem \ref{thm:reg}, thereby improving the current results in the literature concerning the Bernstein problem in $\He$, which require Lipschitzianity as minimal assumption.

\subsection{Organization of the paper.}
The paper is organized as follows. In Section \ref{INTERPOL}, and in particular in Theorem \ref{lem:exlambda}, we construct a horizontally ruled parametric surface spanning a given boundary datum. In Section \ref{GRAPH}, we show that this parametric surface is both a left intrinsic graph (Theorem \ref{thm:sx_graph}) and a right intrinsic graph (Theorem \ref{lem:dx_graph}). Section \ref{CAL} contains the calibration argument, which allows us to demonstrate that this surface is area-minimizing with respect to a suitable class of competitors, see Theorem \ref{thm:E_minimizer_P_H}. Finally, Section \ref{REG} is devoted to the proof of Theorem \ref{thm:reg}, in which we establish a new regularity result for minimizers of $\cA_D$.

\section{Horizontal interpolation of a boundary datum}
\label{INTERPOL}

In this section, we construct a horizontally ruled parametric surface spanning a Lipschitz continuous boundary
datum defined  on the boundary $\de D$ of an open domain $D \subset \W$ of the form
%
%
\begin{equation*}
	D = D_{\gamma_1,\gamma_2} := \{ (y,t) \in \mathbb{W} \, : \, t \in I, \, \gamma_1(t) < y < \gamma_2(t) \} .
\end{equation*}

%
%
\noindent
The notations are those given in the Introduction.

\begin{theorem} \label{lem:exlambda}
	Let $D = D_{\gamma_1,\gamma_2} \subset \mathbb{W}$ and $\phi :\partial D\to \R$ be such that:
	\begin{equation} \label{eq:condlambda}
		\zeta < 1 .
	\end{equation}
	Then there exists a bi-Lipschitz, increasing function $\lambda: \bar{I} \to \bar{I}$ such that:
	\begin{align} \label{eq:proplambda}
		\text{\emph{(i)}}\,\,\, & p_2(\lambda(s)) - p_1(s) \in H_{p_1(s),} \qquad \text{for all $s \in \bar{I}$,} \nonumber \\
		\text{\emph{(ii)}}\,\,\, & p_2(\lambda') - p_1(s) \in H_{p_1(s)} \text{  implies  } \lambda' = \lambda(s), \qquad \text{for all $s,\lambda' \in \bar{I}$,} \nonumber \\
		\text{\emph{(iii)}}\,\,\, & \lambda(0) = 0 \quad \text{and} \quad \lambda(\bar{t}) = \bar{t} , \\
		\text{\emph{(iv)}}\,\,\, & \dfrac{1 - \zeta}{1 + \zeta} \leq \Lip(\lambda) \leq \dfrac{1 + \zeta}{1 - \zeta} \quad \text{and}\quad \dfrac{1 - \zeta}{1 + \zeta} \leq \Lip(\lambda^{-1}) \leq \dfrac{1 + \zeta}{1 - \zeta} . 
		\nonumber
	\end{align}
\end{theorem}

\begin{proof}
	We claim that, for any $s \in \bar{I}$, there exists $\lambda \in \bar{I}$ such that:
	\begin{equation} \label{eq:belongstoHpf}
		p_2(\lambda) - p_1(s) \in H_{p_1(s)}.
	\end{equation}
	Recall that, for $s \in \bar{I}$, $H_{p_1(s)}$ expresses the horizontal distribution at $p_1(s)$, that is, the linear space spanned by
	\begin{equation*}
		X(p_1(s)) = (1, \, 0, \, 2 \gamma_1(s)) \qquad \text{and} \qquad Y(p_1(s)) = (0, \, 1, \, - 2 \phi_1(s)) .
	\end{equation*}
	Then $H_{p_1(s)}$ is uniquely determined by the Cartesian equation
	\begin{equation} \label{eq:carteqH}
		2 \gamma_1(s) x - 2 \phi_1(s) y - t = 0 ,
	\end{equation}
	and, by
	the definition of $p_1$ and $p_2$, we have
	\begin{equation*}
		p_2(\lambda) - p_1(s) = (\phi_2(\lambda) - \phi_1(s), \, \gamma_2(\lambda) - \gamma_1(s), \, \lambda + 2 \gamma_2(\lambda)\phi_2(\lambda) - s - 2 \gamma_1(s)\phi_1(s)) .
	\end{equation*}
	Thus, by \eqref{eq:carteqH}, condition \eqref{eq:belongstoHpf} holds if and only if $\lambda\in \bar I$ is a solution to the equation:
	\begin{equation} \label{eq:belongstoHpfzero} 		
		\cQ_{\phi}(s,\lambda) := \lambda - s + 2 (\gamma_2(\lambda) - \gamma_1(s))(\phi_2(\lambda) + \phi_1(s))
		= 0 .
	\end{equation}
	If $s = 0$ or  $s = \bar{t}$, the solution is  $\lambda = 0$ or $\lambda = \bar{t}$, respectively. 
	When $s \in I$, we  argue by continuity. 
	Note that: 
	\begin{align} \label{eq:values_q_extrem}
		\cQ_{\phi}(s,0) &= -s - 2 \gamma_1(s)(\phi(0,0) + \phi_1(s)) , 
		\\
		\quad \cQ_{\phi}(s,\bar{t}) &= \bar{t} - s - 2 \gamma_1(s)(\phi(0,\bar{t}) + \phi_1(s)) .
	\end{align}
	Since  $\gamma_1(0)= \gamma_1(\bar t)=0$, we deduce that 
	\begin{align} \label{eq:Thm_Zeri_qphi}
		& |2 \gamma_1(s)(\phi(0,0) + \phi_1(s))| \leq 4 \Lip(\gamma) \| \phi \|_{\infty} |s| \leq \zeta |s|, \\
		& |2 \gamma_1(s)(\phi(0,\bar{t}) + \phi_1(s))| \leq 4 \Lip(\gamma) \| \phi \|_{\infty} |s-\bar{t}| \leq \zeta |s - \bar{t}| , \nonumber 
	\end{align}
	and therefore \eqref{eq:condlambda}, \eqref{eq:values_q_extrem}, \eqref{eq:Thm_Zeri_qphi} guarantee that $\cQ_{\phi}(s,0) < 0$ and $\cQ_{\phi}(s,\bar{t}) > 0$. By  the continuity of $\lambda\mapsto \cQ_{\phi}(s,\lambda)$, there exists at least one $\lambda \in I$ solving \eqref{eq:belongstoHpfzero}, and consequently also \eqref{eq:belongstoHpf}. This concludes the proof of Claim (i).

	We prove Claim (ii). Let $s \in \bar{I}$, $\lambda, \lambda' \in \bar{I}$ be such that:
	\begin{equation*}
		\cQ_{\phi}(s,\lambda) = 0 , \qquad \cQ_{\phi}(s,\lambda') = 0 .
	\end{equation*}
	By the  definition of $\cQ_{\phi}$ in \eqref{eq:belongstoHpfzero}, we deduce that:
	\begin{align} \label{eq:lambda-mu}
		\lambda - \lambda' & = 2 \gamma_1(s)(\phi_2(\lambda) - \phi_2(\lambda')) - 2 (\gamma_2(\lambda) - \gamma_2(\lambda')) \phi_1(s) - 2(\gamma_2(\lambda) \phi_2(\lambda) -\gamma_2(\lambda') \phi_2(\lambda')) \nonumber \\
		& = 2 \gamma_1(s)(\phi_2(\lambda) - \phi_2(\lambda')) - 2 (\gamma_2(\lambda) - \gamma_2(\lambda')) \phi_1(s) - \\
		& \qquad \qquad \qquad  - 2 \gamma_2(\lambda)(\phi_2(\lambda) - \phi_2(\lambda')) - 2(\gamma_2(\lambda) - \gamma_2(\lambda')) \phi_2(\lambda') . \nonumber
	\end{align}
	The summands in the identity above can be estimated as follows:
	\begin{align} \label{eq:est|lambda-mu|}
		& 2 |\gamma_1(s)||\phi_2(\lambda) - \phi_2(\lambda')|| \leq 2 \, \| \gamma \|_{\infty} \Lip(\phi) \, |\lambda - \lambda'|, \nonumber \\
		& 2 |\gamma_2(\lambda) - \gamma_2(\lambda')| |\phi_1(s)| \leq 2 \, \Lip(\gamma) \| \phi \|_{\infty} |\lambda - \lambda'|, \\
		& 2 |\gamma_2(\lambda)||\phi_2(\lambda) - \phi_2(\lambda')| \leq 2 \| \gamma \|_{\infty} \Lip(\phi) \, |\lambda - \lambda'|, \nonumber \\
		& 2|\gamma_2(\lambda) - \gamma_2(\lambda')| |\phi_2(\lambda')| \leq 2 \, \Lip(\gamma) \| \phi \|_{\infty} |\lambda - \lambda'| . \nonumber
	\end{align}
	Hence, by assumption
	\eqref{eq:condlambda}, \eqref{eq:lambda-mu} and \eqref{eq:est|lambda-mu|} ensure that $\lambda = \lambda'$. 
	This concludes the proof of Claims (ii) and (iii).


	We are left with the proof of (iv). 
	Let  $s,s' \in I$ and $\lambda := \lambda(s)$, $\lambda' := \lambda(s')$. To  estimate $|\lambda - \lambda'|$, we start from the identity 
	\begin{align} \label{eq:lambda-lambda'}
		0 & = \cQ_{\phi}(s,\lambda) - \cQ_{\phi}(s',\lambda') \nonumber \\
		& =  \lambda - \lambda' + s - s' - 2 \gamma_1(s) (\phi_2(\lambda) + \phi_1(s)) + 2 \gamma_1(s') (\phi_2(\lambda') + \phi_1(s')) + \nonumber \\
		& \qquad \qquad \qquad \qquad \qquad \qquad \quad + 2 \gamma_2(\lambda)(\phi_1(s) + \phi_2(\lambda)) - 2 \gamma_2(\lambda')(\phi_1(s') + \phi_2(\lambda')) \nonumber \\
		& = \lambda - \lambda' + s - s' - 2 (\gamma_1(s) \phi_2(\lambda) - \gamma_1(s') \phi_2(\lambda')) - 2 (\gamma_1(s) \phi_1(s) - \\
		& \qquad - \gamma_1(s') \phi_1(s')) + 2 (\gamma_2(\lambda) \phi_1(s) -  \gamma_2(\lambda') \phi_1(s')) + 2 (\gamma_2(\lambda) \phi_2(\lambda) -  \gamma_2(\lambda')\phi_2(\lambda')) . \nonumber
	\end{align}
	By the  estimates:
	\begin{align} \label{eq:est|lambda-lambda'|}
		& 2 |\gamma_1(s) \phi_2(\lambda) - \gamma_1(s') \phi_2(\lambda')| \leq 2 \, \Lip(\gamma) \| \phi \|_{\infty} |s-s'| + \nonumber \\
		& \qquad \qquad \qquad \qquad \qquad \qquad \qquad \qquad \qquad + 2 \| \gamma \|_{\infty} \Lip(\phi) \, |\lambda - \lambda'|, \nonumber \\
		& 2 |\gamma_1(s) \phi_1(s) - \gamma_1(s') \phi_1(s')| \leq 2 \, \Lip(\gamma) \| \phi \|_{\infty} |s-s'| + \nonumber \\
		& \qquad \qquad \qquad \qquad \qquad \qquad \qquad \qquad \qquad + 2 \| \gamma \|_{\infty} \Lip(\phi) \, |s-s'|, \\
		& 2 |\gamma_2(\lambda) \phi_1(s) -  \gamma_2(\lambda') \phi_1(s')| \leq 2 \, \Lip(\gamma) \| \phi \|_{\infty} |\lambda - \lambda'| + \nonumber \\
		& \qquad \qquad \qquad \qquad \qquad \qquad \qquad \qquad \qquad + 2 \| \gamma \|_{\infty} \Lip(\phi) \, |s - s'|, \nonumber \\
		& 2 |\gamma_2(\lambda) \phi_2(\lambda) -  \gamma_2(\lambda')\phi_2(\lambda')| \leq 2 \, \Lip(\gamma) \| \phi \|_{\infty} |\lambda - \lambda'| + \nonumber \\ 
		& \qquad \qquad \qquad \qquad \qquad \qquad \qquad \qquad \qquad + 2 \| \gamma \|_{\infty} \Lip(\phi) \, |\lambda - \lambda'|  , \nonumber 
	\end{align}
	from
	\eqref{eq:lambda-lambda'} and \eqref{eq:est|lambda-lambda'|} we deduce that
	\begin{equation*}
		(1 - \zeta)|\lambda - \lambda'| \leq (1 + \zeta) |s - s'| \qquad \text{and} \qquad (1 + \zeta)|\lambda - \lambda'| \geq (1 - \zeta) |s - s'| .
	\end{equation*}
	In particular, we have
	\begin{equation*}
		\Lip(\lambda) \leq \dfrac{1 + \zeta}{1 - \zeta} \qquad \text{and} \qquad  \Lip(\lambda) \geq \dfrac{1 - \zeta}{1 + \zeta} .
	\end{equation*}
	This concludes the proof.

	%
	
\end{proof}

\begin{remark} \label{rem:smooth_lambda}
	The function $\lambda$ introduced in   Theorem \ref{lem:exlambda}
	has the same regularity of  $\phi$ and $\gamma_1,\gamma_2$.
	Namely,  if  $\phi_1, \phi_2, \gamma_1,\gamma_2 \in C^k(I)$  then $\lambda \in C^k(I)$, $k\in\N$. This is an immediate consequence of the Implicit Function Theorem applied to $\cQ_{\phi}(s,\lambda)$. In fact,  for $\zeta<1$ we have
	\begin{equation*}
		\dfrac{\de}{\de \lambda} \, \cQ_{\phi}(s,\lambda) = 1 + 2 \gamma_2'(\lambda)(\phi_2(\lambda) + \phi_1(s)) + 2(\gamma_2(\lambda) - \gamma_1(s)) \phi_2'(\lambda) \geq 1 - \dfrac{\zeta}{2} > 0 .
	\end{equation*}
\end{remark}

\begin{remark} \label{rem:approx_lambda}
	For all $j \in\N$, let $\phi^j : \de D \to \R$ be continuous functions such that, denoting by $\phi_1^j(s) := \phi^j(\gamma_1(s),s)$ and $\phi_2^j(s) := \phi^j(\gamma_2(s),s)$, $\phi^j_1$ and $\phi^j_2$ are Lipschitz in $I$, and
	\begin{equation} \label{eq:approx_phi}
		\phi_1^j , \, \phi_2^j \in C^{1}(I) \qquad \text{and}
		\qquad \| \phi_1^j - \phi_1 \|_{\infty}, \, \| \phi_2^j - \phi_2 \|_{\infty} \to 0, 
	\end{equation}
	Also assume that $\gamma_1,\gamma_2\in C^1(I)$.
	
	By Remark \ref{rem:smooth_lambda}, 
	the function
	$\lambda^j$  solving  $\cQ_{\phi^j}(s,\lambda^j(s)) = 0$, $s \in \bar{I}$,
	satisfies 
	$\lambda^j\in C^1(I)$, for all $j \in\N$. 
	Moreover,  $\lambda^j \to \lambda$ uniformly on $I$. Indeed, using  $\cQ_{\phi}(\lambda(s),s) - \cQ_{\phi^j}(\lambda^j(s),s) = 0$, for $s \in I$, and reasoning as in the proof of Theorem \ref{lem:exlambda}, we obtain the estimates:
	\begin{equation*}
		(1 - \zeta) |\lambda^j(s) - \lambda(s)| \leq 4 \| \gamma \|_{\infty} (\| \phi_1^j - \phi_1 \|_{\infty} + \| \phi_2^j - \phi_2 \|_{\infty}) .
	\end{equation*}
	Provided $\zeta < 1$, this implies that $\| \lambda^j - \lambda \|_{\infty} \to 0$.
\end{remark}

Under the assumption $\zeta<1$, let $\lambda:\bar I \to \bar I$ be the function given by Theorem \ref{lem:exlambda}.
On the domain $Q := (0,1) \times I$, consider the function $\rho : Q \to \He$:
\begin{equation*}
	\rho(h,s) = (1-h) \, p_1(s) + h \, p_2(\lambda(s)) ,
\end{equation*}
where $p_1$ and $p_2$ are as in \eqref{eq:def_p_1} and \eqref{eq:def_p_2}, respectively. 
The components of $\rho = (\rho_1,\rho_2,\rho_3)$ have the following form, for $(h,s) \in \bar{Q}$:
\begin{align} \label{eq:def_rho}
	\rho_1(h,s) & := (1-h)\phi_1(s) + h \, \phi_2(\lambda(s)), \nonumber \\
	\rho_2(h,s) & := (1 - h) \gamma_1(s) + h \gamma_2(\lambda(s)), \\
	\rho_3(h,s) & := (1-h) s + h \lambda(s) + 2(1 - h) \gamma_1(s) \phi_1(s) + 2 h \gamma_2(\lambda(s)) \phi_2(\lambda(s)) . \nonumber
\end{align}
The set $R_{\phi} := \rho(Q)\subset \He$ is a (the) horizontally ruled, parametric surface that spans the intrinsic graph of $\phi$.

\section{Intrinsic graph structure of the horizontally ruled surface $R_{\phi}$}
\label{GRAPH}

Let $D\subset\mathbb W$ and $\phi:\partial D \to \R$ be as in the previous section.
We prove that  for small $\zeta<1$ the parametric surface $R_{\phi}$ introduced via \eqref{eq:def_rho} is the the left intrinsic graph of a function $u:D\to\R$ such that $u = \phi$ on $\de D$. 
We will   show that it is also  a right intrinsic graph on a suitable domain in $\mathbb W$.

In this section,  $D = D_{\gamma_1,\gamma_2} \subset \mathbb{W}$ is a convex domain, i.e., $\gamma_1$ is convex and $\gamma_2$ is concave. We call such a domain a \emph{lenticular domain}.

\subsection{Left intrinsic graph.}
We start by analyzing the left intrinsic graph structure of $R_{\phi}$. The set $R_{\phi} = \rho(Q)$ is introduced via \eqref{eq:def_rho}.
\begin{theorem} \label{thm:sx_graph}
	Let $D = D_{\gamma_1,\gamma_2} \subset \mathbb{W}$ be a lenticular domain and let $\phi : \de D \to \R$ be such that
	\begin{equation} \label{eq:cond_invert}
		\zeta < \dfrac{\sqrt{129} - 11}{4}.
	\end{equation}
	Then there exists a function $u: \bar{D} \to \R$ with the following properties:
	\begin{itemize}
		\item[(i)] $u$ represents $R_{\phi}$ as a left intrinsic graph on D, i.e.,
		\begin{equation} \label{eq:intr_graph_R}
			R_{\phi} = S_u = \{ (u(y,t), \, y, \, t + 2 y u(y,t)) \in \He \, : \, (y,t) \in D \} ;
		\end{equation}
		\item[(ii)] $u$ is Lipschitz continuous in $\bar{D}$;
		\item[(iii)] $u = \phi$ on $\de D$.
	\end{itemize} 
\end{theorem}

\begin{proof} 
	Finding a function $u$ realizing \eqref{eq:intr_graph_R} is equivalent to proving  that the projection onto $\mathbb W$ along the integral lines of  $X$ establishes a $1:1$ correspondence of $R_{\phi}$ and $D$. Let $\pi:\mathbb H^1 = \R^3 \to\mathbb W =\R^2$ be this projection, i.e.,
	\begin{equation} \label{eq:def_proj_pi}
		\pi(x,y,t) = (y, t-2 xy) , \qquad \text{for all $(x,y,t) \in \He$.}
	\end{equation}
	and let $F=(F_1,F_2) : Q\to \R^2$ be the composition $F=\pi\circ \rho$, where  $\rho$ is the parameterization of $R_{\phi}$ introduced  in \eqref{eq:def_rho}:
	\begin{align}
		F_1(h,s) & := \rho_2(h,s) = (1 - h) \gamma_1(s) + h \gamma_2(\lambda(s)) \label{eq:F_1} \\
		F_2(h,s) & := \rho_3(h,s) - 2 \rho_1(h,s) \rho_2(h,s) \label{eq:F_2} \\
		& = (1 - h) s  + h \lambda(s)  + 2 h (1-h)\cE(s) .
		\nonumber
	\end{align}
	Here and hereafter, we use the short notation
	\begin{equation} \label{E}
		\cE(s) = 
		[\gamma_2(\lambda(s)) - \gamma_1(s)] [\phi_2(\lambda(s)) - \phi_1(s)]  . 
	\end{equation}
	\medskip
	
	\emph{Step 1.} We claim that  $F(Q) = D$. Let us first prove the inclusion $D \subset F(Q)$. 
	Given $(y,t) \in D$, we look for $(h,s) \in Q$ satisfying
	\begin{equation} \label{eq:eq_F_1_2}
		F_1(h,s) = y \qquad \text{and} \qquad F_2(h,s) = t  .
	\end{equation}
	By the first equation in \eqref{eq:eq_F_1_2}, the solution  $h = h_y(s)$ is given by
	\begin{equation} \label{eq:def_h_y}
		h_y(s) := \dfrac{y - \gamma_1(s)}{\gamma_2(\lambda(s)) - \gamma_1(s)} .
	\end{equation}
	In order to have $(h,s) \in Q$, we need to impose that $0 < h_y(s) < 1$, which is equivalent to $s\in J_y$, where
	\begin{equation*}
		J_y := \{ s \in I \, : \, \gamma_1(s) < y < \gamma_2(\lambda(s)) \} .
	\end{equation*}
	As $\gamma_1 < 0$ and $\gamma_2 > 0$ in $I$, the interval $J_y$ is
	\begin{equation}\label{eq:J_y_redef}
		\begin{split} 
			& J_y := \{ s \in I \, : \, \gamma_1(s) < y \} \qquad \qquad \text{when $y \leq 0$,} , \\
			& J_y := \{ s \in I \, : \, \gamma_2(\lambda(s)) > y \} , \qquad \, \, \, \, \text{when $y > 0$.}
		\end{split}
	\end{equation}
	Notice that   $J_y\neq \emptyset$ because $(y,t) \in D$. 
	By the convexity of $\gamma_1$, the concavity of $\gamma_2$ and the monotonicity of $\lambda$, $J_y$  is an open interval, i.e., there exist $s_y^-, s_y^+ \in \bar{I}$, with $s_y^- < s_y^+$, such that $J_y = (s_y^-,s_y^+)$. From \eqref{eq:J_y_redef}, it follows that
	\begin{equation} \label{eq:deJ_y}
		\begin{split}		& \gamma_1(s_y^-) = y = \gamma_1(s_y^+) , \qquad \qquad \, \, \, \, \text{when $y \leq 0$,} \\
			& \gamma_2(\lambda(s_y^-)) = y = \gamma_2(\lambda(s_y^+)) , \qquad \text{when $y > 0$.} \nonumber
		\end{split}
	\end{equation}
	We  look for a solution $s=s_t \in J_y$ to the equation  $\tau_y(s) = t$, where
	\begin{align} \label{eq:def_tau_y}
		\tau_y(s) & := F_2(h_y(s),s)) = \left( 1 - h_y(s) \right) s \, + h_y(s) \lambda(s) \, + \, 2 h_y(s) \left( 1 - h_y(s) \right)  \mathcal E(s).
	\end{align}
	We claim that
	\begin{equation} \label{eq:-<t<+}
		\tau_y(s_y^-) < t < \tau_y(s_y^+) .
	\end{equation}
	If this holds, the existence of a solution follows by a continuity argument. There are two possibilities, depending on the sign of $y$:
	\begin{itemize}
		\item[(1)] We have $y \leq 0$. Then \eqref{eq:deJ_y} yields 
		$\gamma_1(s_y^-) = y = \gamma_1(s_y^+)$ and $h_y(s_y^-) = 0 = h(s_y^+)$. Hence, there holds
		\begin{equation} \label{eq:tau(J)_yleq0}
			\tau_y(s_y^-) = s_y^- \qquad \text{and} \qquad \tau_y(s_y^+) = s_y^+ .
		\end{equation}
		Since  $\gamma_1$ is convex and $\gamma_1(s_{y}^{\pm}) = y > \gamma_1(t)$, we have 
		$s_y^- < t < s_y^+$, and  \eqref{eq:-<t<+} follows.
		
		\item[(2)] We have $y > 0$. In this case, \eqref{eq:deJ_y} gives   $\gamma_2(\lambda(s_y^-)) = y = \gamma_2(\lambda(s_y^+))$, and  thus $h_y(s_y^-) = 1 = h(s_y^+)$. It follows that 
		\begin{equation} \label{eq:tau(J)_y>0}
			\tau_y(s_y^-) = \lambda(s_y^-) \qquad \text{and} \qquad \tau_y(s_y^+) = \lambda(s_y^+) .
		\end{equation}
		As in  the previous case, we have $\lambda(s_y^-) < t < \lambda(s_y^+)$, because $\gamma_2$ is concave and $\gamma_2(\lambda(s^{\pm}_y)) = y < \gamma_2(t)$. This implies \eqref{eq:-<t<+}.
	\end{itemize}
	This concludes the proof of the inclusion $D \subset F(Q)$.
	
	We next show that $F(Q) \subset D$. Let $(y,t) \in F(Q)$, and assume for instance that $y \leq 0$. 
	We claim
	that $\gamma_1(t) < y$ (i.e., $ (y,t)\in D$).
	
	There exists $(h,s)\in Q$ such that $(y,t) = F(h,s)$. In particular,  $s \in J_y$ is a solution to $t=\tau_y(s) = F_2(h,s)$ (see \eqref{eq:def_tau_y})  and, 
	by the definition of $F_1$ we have $y = (1 - h) \gamma_1(s) + h \gamma_2(\lambda(s))$. 
	Therefore, the condition $\gamma_1(t) < y$ is equivalent to: 
	\begin{align} \label{eq:equiv_gamma_1<y_more_gen}
		\gamma_1(F_2(h,s)) < (1 - h) \gamma_1(s) + h \gamma_2(\lambda(s)) , \qquad \text{for $0 < h < 1$, $s \in I$,}
	\end{align}
	where $F_2(h,s) = (1-h)s + h \lambda(s) + 2h(1-h) \cE(s) $.
	
	We check \eqref{eq:equiv_gamma_1<y_more_gen}. By the convexity of $\gamma_1$, there holds:
	\begin{align*}
		\gamma_1(F_2(h,s)) & \leq \gamma_1((1-h)s + h\lambda(s)) + 2h(1-h)\Lip(\gamma_1)
		|\cE(s)| \\
		& \leq (1-h) \gamma_1(s) + h \gamma_1(\lambda(s)) + 2h(1-h) \Lip(\gamma_1) |\cE(s)| .
	\end{align*}
	Let us define
	\begin{align*}
		&  \alpha(h) := (1-h) \gamma_1(s) + h \gamma_1(\lambda(s)) + 2h(1-h) \Lip(\gamma_1)  |\cE(s)| \\
		& \beta(h) := (1-h) \gamma_1(s) + h \gamma_2(\lambda(s))  .
	\end{align*}
	Proving \eqref{eq:equiv_gamma_1<y_more_gen} reduces to proving that 
	$\alpha(h) < \beta(h)$ for all   $h\in(0,1)$. By the definition of $\cE$ in \eqref{E}, 
	this inequality is equivalent to:
	\begin{equation} \label{eq:alpha'<beta'}
		2(1-h) \Lip(\gamma_1) (\gamma_2(\lambda(s)) - \gamma_1(s))|\phi_2(\lambda(s)) - \phi_1(s)| < \gamma_2(\lambda(s)) - \gamma_1(\lambda(s)) .
	\end{equation}
	This is in turn implied by the   inequality: 
	\begin{equation} \label{eq:alpha'<beta'_stronger}
		\Lip(\gamma_1)  \dfrac{\gamma_2(\lambda) - \gamma_1}{\gamma_2(\lambda) - \gamma_1(\lambda)} |\phi_2(\lambda) - \phi_1| \leq \dfrac{1}{2} .
	\end{equation}
	The ratio appearing on the left hand side is estimated as follows:
	\begin{align} \label{eq:est_ratio_LHS}
		\dfrac{\gamma_2(\lambda) - \gamma_1}{\gamma_2(\lambda) - \gamma_1(\lambda)} & \leq 1 + \dfrac{|\gamma_1(\lambda) - \gamma_1|}{\gamma_2(\lambda) - \gamma_1(\lambda)} \nonumber \\
		& \leq 1 + \dfrac{|\gamma_1(\lambda) - \gamma_1|}{|\gamma_1(\lambda)|} \\
		& \leq 2 + \dfrac{|\gamma_1|}{|\gamma_1(\lambda)|} . \nonumber
	\end{align}

	Now we claim that:
	\begin{equation} \label{eq:gamma_1/gamma_1(lambda)}
		\dfrac{|\gamma_1|}{|\gamma_1(\lambda)|} \leq \Lip(\lambda^{-1}) .
	\end{equation}
	There are two possibilities:
	\begin{itemize}
		
		\item[(1)] If $\lambda(s) \leq s$ there exists   $0 \leq h \leq 1$ such that $\lambda(s) = hs$. The concavity of $|\gamma_1|$ yields $|\gamma_1(\lambda(s))| \geq h |\gamma_1(s)|$. Moreover, by $\lambda(0) = 0$ we have:
		\begin{equation*}
			hs = \lambda(s) \geq \dfrac{s}{\Lip(\lambda^{-1})} , 
		\end{equation*}
		i.e., $h \geq \Lip(\lambda^{-1})^{-1}$.
		
		\item[(2)] If $\lambda(s) > s$ there exists $0 < h < 1$ such that $\lambda(s) = hs + (1 - h) \bar{t}$. Since $|\gamma_1|$ is concave and $\gamma_1(\bar t)=0$,  we obtain $|\gamma_1(\lambda(s))| \geq h |\gamma_1(s)|$ and thus, by $\lambda(0) = 0$, we have:  
		\begin{equation*}
			h(\bar{t}-s) = \bar{t} - \lambda(s) \geq \dfrac{\bar{t} - s}{\Lip(\lambda^{-1})} , 
		\end{equation*}
		i.e., $h \geq \Lip(\lambda^{-1})^{-1}$. 
	\end{itemize}
	In both cases  \eqref{eq:gamma_1/gamma_1(lambda)} follows:
	\begin{equation*}
		\dfrac{|\gamma_1|}{|\gamma_1(\lambda)|} \leq \dfrac{1}{h} \leq \Lip(\lambda^{-1}).
	\end{equation*}
	
	By \eqref{eq:est_ratio_LHS}, \eqref{eq:gamma_1/gamma_1(lambda)},  and   the estimate for $\Lip(\lambda^{-1})$   in \eqref{eq:proplambda}, we get the following upper-bound for the left hand side of \eqref{eq:alpha'<beta'_stronger}:
	\begin{align} \label{eq:est(alpha',beta')}
		\Lip(\gamma_1)		 \dfrac{\gamma_2(\lambda) - \gamma_1}{\gamma_2(\lambda) - \gamma_1(\lambda)} |\phi_2(\lambda) - \phi_1| & \leq 
		2 \Lip(\gamma_1)	 \| \phi \|_{\infty} \left( 2 + \Lip(\lambda^{-1}) \right) \nonumber \\
		& 
		\leq 2 \Lip(\gamma_1)	 \| \phi \|_{\infty} \left( 2 + \dfrac{ 1+ \zeta}{1 - \zeta} \right) \nonumber \\
		& \leq \dfrac{\zeta}{2} \left( 2 + \dfrac{ 1 + \zeta}{1 - \zeta} \right).
	\end{align}
	Assumption \eqref{eq:cond_invert} and \eqref{eq:est(alpha',beta')} imply \eqref{eq:alpha'<beta'_stronger}, and then also
	the claim $\gamma_1(t) < y$. This concludes the proof of  $F(Q) \subset D$ when $y\leq0$.
	The argument when $y>0$ is completely similar and can be omitted.
	\medskip
	
	\emph{Step 2.} We claim that $F: Q\to D$ is injective.  Assume that $F(h_1,s_1) = (y,t) = F(h_2,s_2)$ for some $(h_1,s_1), (h_2,s_2) \in Q$. 
	Using the notation in the previous step, this   implies that
	\begin{equation} \label{eq:conseq_noinj}
		s_1,s_2 \in J_y , \qquad h_1 = h_y(s_1) \quad \text{and} \quad h_2 = h_y(s_2) , \qquad \tau_y(s_1) = t = \tau_y(s_2) .
	\end{equation}
	We claim that the function $\tau_y$ is monotonically increasing in $J_y$. If this is true,  $\tau_y(s_1) = \tau_y(s_2) $ implies $s_1=s_2$ and hence also $h_1=h_2$.
	
	Since $\lambda$ is bi-Lipschitz and $\phi_1$, $\phi_2$, $\gamma_1$, $\gamma_2$ are Lipschitz, for almost every $s \in J_y$ the functions $\lambda$, $\gamma_1$, $\gamma_2 \circ \lambda$, $\phi_1$, $\phi_2 \circ \lambda$ are simultaneously differentiable at $s$.
	At these points, we can compute the derivative of $\tau_y$, see \eqref{eq:def_tau_y}:
	\begin{equation}\label{eq:tau'}
		\begin{split} 
			\tau_y'(s) & = 1 - h_y - s \, h_y' + \lambda' \, h_y + \lambda \, h_y' \\
			& \qquad \qquad   + 2 h_y'(1 - 2 h_y) \, \cE + 2 h_y(1 - h_y) \cE' \\
			& = (1 - h_y) (1 + 2 h_y \, \cE') + h_y \lambda' \\
			&\qquad \qquad   + h_y' (\lambda - s + 2(1 - 2 h_y)  \cE).
		\end{split}
	\end{equation}
	Since $\lambda$ solves the equation $\cQ_{\phi}(s,\lambda) = 0$, recalling the definition of $h_y$ in \eqref{eq:def_h_y} and of $\rho(h,s)$ in \eqref{eq:def_rho}, we infer that
	\begin{equation} \label{eq:lambda-s}
		\begin{split}
			\lambda - s  \,+ & \, 2(1 - 2 h_y) \, \cE  = - 2 (\gamma_2(\lambda) - \gamma_1) \big[ \phi_2(\lambda) + \phi_1 - \left( 1 - 2 h_y \right)
			(\phi_2(\lambda) - \phi_1) \big] \\
			& = - 2 (\gamma_2(\lambda) - \gamma_1) \, [ \, 2 \phi_1 + 2h_y (\phi_2(\lambda) - \phi_1) ] \\
			& = - 4 (\gamma_2(\lambda) - \gamma_1) \rho_1(h_y,s)  .
		\end{split}
	\end{equation}
	Now a   computation yields
	\begin{align} \label{eq:h'/h}
		h_y'(s) = - \dfrac{\gamma_1'}{\gamma_2(\lambda) - \gamma_1}(1 - h_y) - \gamma_2'(\lambda) \lambda' \dfrac{h_y}{\gamma_2(\lambda) - \gamma_1},
	\end{align}
	and then,   by \eqref{eq:lambda-s} and  \eqref{eq:h'/h}, we get
	\begin{align} \label{eq:h'/h*}
		h_y' (\lambda - s + 2(1 - 2 h_y) \, \cE) = 4 \rho_1(h_y,s) [\gamma_1'(1 - h_y) + \gamma_2'(\lambda) \lambda' h_y] .
	\end{align}
	Inserting \eqref{eq:h'/h*} inside \eqref{eq:tau'}, we obtain the following identity:
	\begin{equation} 
		\label{eq:tau_y'_fin}
		\begin{split}
			\tau_y'(s) = (1-h_y) \big ( 1 + 2 h_y  \cE' + 4 \rho_1(h_y,s) \gamma_1'   \big)
			+ h_y \, \lambda' \big (1 + 4 \rho_1(h_y,s) \gamma_2'(\lambda) \big ) .
		\end{split}
	\end{equation}
	The derivative of $\cE$ is estimated in the following way: 
	\begin{equation}\label{cE}
		\begin{split}
			|\cE'(s)| &\leq 2 \| \phi \|_{\infty}\big(\Lip(\gamma_2) \Lip(\lambda) + \Lip(\gamma_1)\big) +
			\\
			& \qquad \qquad \qquad + 2 \| \gamma \|_{\infty} \big(\Lip(\phi_2) \Lip(\lambda) + \Lip(\phi_1)\big) 
			\\
			&\leq 2 \left( 1 + \Lip(\lambda) \right)\big  (\Lip(\gamma) \| \phi \|_{\infty} + \| \gamma \|_{\infty} \Lip(\phi)\big) 
			\\
			& \leq \dfrac{\zeta}{2} \left( 1 + \dfrac{1 + \zeta}{1 - \zeta} \right) ,
		\end{split}
	\end{equation}
	and we also have
	\begin{equation}\label{rho}
		|\rho_1(h_y(s))| \leq \| \phi \|_{\infty} .
	\end{equation}
	We deduce that 
	\begin{equation} \label{eq:est_E'_F}
		|2 h_y \, \cE' + 4 \rho_1(h_y,s) \gamma_1'|  \leq \zeta \left( 2 + \dfrac{1 + \zeta}{1 - \zeta} \right) , \, \, \, \,
		|4 \rho_1(h_y,s) \gamma_2'(\lambda)| \leq \zeta \leq \zeta \left( 2 + \dfrac{1 + \zeta}{1 - \zeta} \right) .
	\end{equation}
	By \eqref{eq:est_E'_F}, \eqref{eq:tau_y'_fin}, and the bounds   for $\Lip(\lambda)$ in \eqref{eq:proplambda}, we get
	the following lower-bound for $\tau_y'(s)$:
	\begin{equation}\label{TAU}
		\tau_y'(s) \geq \left[ 1 - \zeta \left( 2 + \dfrac{1 + \zeta}{1 - \zeta} \right) \right] \dfrac{1 - \zeta}{1 + \zeta}.
	\end{equation}
	By assumption \eqref{eq:cond_invert}, $\tau_y'$ is uniformly positive in $J_y$, and thus $\tau_y$ is strictly increasing.
	\medskip
	
	\emph{Step 3.}  By the previous step, $F = \pi \circ \rho  : Q \to D$ is invertible and, in particular,  the projection $\pi: R_{\phi}\to D$ is bijective. 
	This proves that $R_{\phi} =\rho (F^{-1} (D))= \pi^{-1}(D)$ is a left intrinsic graph and the graph-function	of Theorem \ref{thm:sx_graph} is $u:D\to\R$ given by
	\begin{equation} \label{eq:def_phi_R}
		u(y,t) = (\rho_1 \circ F^{-1}) (y,t) , \qquad \text{for $(y,t) \in D$.}
	\end{equation}
	\medskip
	
	\emph{Step 4.} We claim  that $u: D\to\R $ is Lipschitz continuous. 
	The argument is long but natural: in order to use the Inverse Function Theorem we  approximate $F$ with $C^1$ diffeomorphisms with uniformly Lipschitz inverse.
	Let $\phi^j : \de D \to \R$, $j\in\N$, be a sequence of functions such that $\phi_1^j(s) =\phi^j(\gamma_1(s),s),\phi_2^j(s) =\phi^j(\gamma_2(s),s) \in C^1(I)$ satisfy\eqref{eq:approx_phi} and  
	\begin{equation*}
		\Lip(\phi^j_1) \leq \Lip(\phi_1) \qquad \text{and} \qquad \Lip(\phi^j_2) \leq \Lip(\phi_2) , \qquad \text{for all $j \geq 1$.}
	\end{equation*}
	Without loss of generality, we can also assume that $\gamma_1,\gamma_2\in C^1(\bar I)$ (otherwise we regularize also $\gamma_1$ and $\gamma_2$).
	In particular, denoting by
	\begin{equation} \label{eq:zeta^j}
		\zeta_j := 4 (\| \gamma \|_{\infty} + \Lip(\gamma)) (\| \phi^j \|_{\infty} + \Lip(\phi^j))  ,
	\end{equation}
	we have 
	$\zeta_j \leq \zeta$ for all $j \geq 1$.
	
	Let  $F^j: Q \to \R^2$ be the function whose components $F_1^j$, $F_2^j$ are defined respectively
	by \eqref{eq:F_1} and \eqref{eq:F_2} with $\phi_1^j,\phi_2^j,\lambda_j$ in place of $\phi_1$, $\phi_2$,  $\lambda$ in \eqref{eq:def_rho}. 
	Remark \ref{rem:smooth_lambda} guarantees that $F^j \in C^1(Q)$, for any $j \geq 1$.
	The partial derivatives of $F_1^j$, $F_2^j$ are:
	\begin{equation}  \label{eq:par_der_F}
		\begin{split}
			& \de_h F_1^j (h,s) = \Delta_j(s),\\
			&   \de_s F_1^j (h,s) = \gamma_1'(s) + h \Delta_j'(s)  ,
			\\
			& \de_h F_2^j (h,s) = - 4 \, \Delta_j(s)  \rho_1^j(h,s) ,\\
			& \de_s F_2^j (h,s) = 1 + h(\lambda_{j}'(s) - 1) + 2h(1-h) \cE_j'(s), 
		\end{split}
	\end{equation}
	where we let
	\begin{equation}
		\label{eq:notations_Delta}
		\begin{split} 
			& \Delta_j(s) := \gamma_2(\lambda_j(s)) - \gamma_1(s) ,
			\\
			&		\cE_j(s) :=  
			[\gamma_2(\lambda_j(s)) - \gamma_1(s)] [\phi_2^j(\lambda_j(s)) - \phi^j_1(s)]  ,
			\\
			& \rho_1^j(h,s) := (1 - h) \phi_1^j(s) + h \phi_2^j(\lambda_j(s)).
		\end{split}
	\end{equation}
	The third line in \eqref{eq:par_der_F} is obtained using the identity $\cQ_{\phi^j}(s,\lambda_j(s))=0$, see \eqref{eq:belongstoHpfzero}.
	
	We claim that the sequence $(F^j)_{j\in\N}$  satisfy the following properties:
	\begin{itemize}
		\item[(i)] $\displaystyle \sup_{j\in\N } \Lip(F^j,Q) < \infty$;
		\item[(ii)]  $F^j : Q \to D$ is invertible;
		\item[(iii)]  $\| F^j - F \|_{\infty} \to 0$;
		\item[(iv)]  $F^j$, $F$ can be extended to $\bar{Q}$ with $F^j(\de Q) \subset \de D$, for $j \in\N$.
	\end{itemize}
	
	Claim (i) follows from  the formulas in \eqref{eq:par_der_F} and from  $\zeta_j \leq \zeta$. Claim (ii) is proved by the same argument as in Step 2.
	Claim (iii) is a consequence of Remark  \ref{rem:approx_lambda}.
	Finally, (iv) is immediate.
	
	By Lemma \ref{lem:appendix} proved below,
	we deduce that, denoting by $\widetilde{F}^j:D\to Q$ the inverse of $F^j$, we have $\widetilde{F}^j \to F^{-1}$ pointwise.
	We claim that the functions $u_j: D\to \R$,
	\begin{equation} \label{eq:def_phi_R^j}
		u_j(y,t) = (\rho_1^j \circ\widetilde{F}^j) (y,t) , \qquad \text{$(y,t) \in D$,}
	\end{equation}
	satisfy:
	\begin{equation}
		\label{LIP}		
		\sup_{j \in \N } \Lip(u_j) < \infty.
	\end{equation}
	Since  $u_j\to u$ pointwise in $D$, it will follow that $\Lip(u) < \infty$,   concluding the proof.
	
	The determinant of the Jacobian matrix 
	of $F^j$ at $(h,s)\in Q $ has the following expression:
	\begin{equation} \label{eq:det_F^j}
		\delta_j(h,s) :=  \det JF^j(h,s)= \Delta_j(s) (1 + k_j(h,s))  ,
	\end{equation}
	where, by \eqref{eq:par_der_F},
	\begin{equation} \label{eq:u^j_det}
		k_j(h,s) := h (\lambda_j'(s) - 1) + 2h(1-h)\cE_j'(s)+ 4 \rho_1^j(s)(\gamma_1'(s) + h  \Delta_j'(s)).
	\end{equation}
	Using the estimate (iv) in \eqref{eq:proplambda}    for $\lambda_j$ and arguing as in \eqref{cE}, recalling that $\zeta_j \leq \zeta$, we obtain:
	\begin{align}
		\label{E1}
		&h  |\lambda_j'(s) - 1| \leq \dfrac{2 \zeta}{1 - \zeta},
		\\ 
		\label{E2}
		& |\cE'(s)|\leq    \zeta \left(1 + \dfrac{1 + \zeta}{1 - \zeta} \right),
		\\
		\label{E3}
		&4 \rho_1^j(s)(\gamma_1'(s) + h  \Delta_j'(s)) \leq 4 \| \phi^j \|_{\infty} \Lip(\gamma)\Lip(\lambda_j) \leq \zeta \, \dfrac{1 + \zeta}{1 - \zeta} .
	\end{align}
	By \eqref{E1}-\eqref{E3} and assumption  \eqref{eq:cond_invert}, we deduce that
	\begin{equation} \label{eq:est_u^j_det}
		|k_j(h,s)| \leq \zeta \left[ \dfrac{2}{1 - \zeta} + \left( 1 + 2 \dfrac{1 + \zeta}{1 - \zeta} \right) \right] <\frac12,
	\end{equation}
	and this implies that, for $(h,s)\in Q$, we have
	\begin{equation}
		\label{DELTA}
		\delta_j(h,s)\geq \frac 12 \Delta_j(h,s)>0.
	\end{equation}
	By the Inverse Function Theorem, we have  $\widetilde{F}^j\in C^1(D;Q)$ and $J  \widetilde{F}^j(y,t) =( J F^j(h,s))^{-1}$,   $(y,t) = F^j(h,s)$.
	The partial derivatives of $\widetilde F^j$ can be therefore computed starting from \eqref{eq:par_der_F}:
	\begin{align}
		\de_y \widetilde{F}^j_1 (y,t) &= \dfrac{1}{\delta_j(h,s)} \big \{ 1 + h(\lambda_{j}'(s) - 1) + 2h(1-h)\cE_j'(s) \big\},
		\nonumber \\
		\de_t \widetilde{F}^j_1 (y,t) &= \dfrac{-1}{\delta_j(h,s)} (\gamma_1'(s) + h\Delta_j'(s)), \nonumber \\
		\de_y \widetilde{F}^j_2 (y,t) &= \dfrac{4}{\delta_j(h,s)} \Delta_j(s) \rho_1^j(h,s) \label{eq:de_F^-1_y_2}, \\
		\de_t \widetilde{F}^j_2 (y,t)  &= \dfrac{1}{\delta_j(h,s)} \Delta_j(s) . \label{eq:de_F^-1_t_2}
	\end{align}
	By the chain rule, the partial derivatives of $u_j$ at $(y,t) = F^j(h,s)$ are:
	\begin{align} \label{eq:de_y_phi^j_R}
		\de_y u_j (y,t) & =\de_h \rho_1 ^j (h,s)  \de_y \widetilde{F}^j_1 (y,t) +
		\de_s \rho_1 ^j (h,s)
		\de_y \widetilde{F}^j_2 (y,t), 
		\\
		\de_t u_j (y,t)  & =\de_h \rho_1 ^j (h,s)   \de_t \widetilde{F}^j_1 (y,t) +\de_s \rho_1 ^j (h,s)  \de_t \widetilde{F}^j_2 (y,t)  ,
		\label{eq:de_t_phi^j_R}
	\end{align}
	where $\de_h \rho_1 ^j (h,s) =  \phi_2^j(\lambda_j(s)) - \phi_1^j(s):=  \psi_j(s)$
	and $	\de_s \rho_1 ^j (h,s)=    [\de_s \phi_1^j(s)  + h \psi_j'(s)]$.
	
	By \eqref{DELTA}, we obtain:
	\begin{equation}
		\begin{split}
			|\de_y \widetilde{F}^j_2 (y,t) | & = \frac{4\Delta_j(y,t)}{\delta_j(y,t)}|\rho_1^j(h,s) | \leq 8\|\phi^j\|_{\infty}\leq 4 \zeta_j\leq 4,
			\\
			|\de_t \widetilde{F}^j_2 (y,t) | &= \frac{\Delta_j(y,t)}{\delta_j(y,t) }\leq 2.
		\end{split}
	\end{equation}
	We claim that  $\psi_j(s) \de_y \widetilde{F}^j_1 (y,t)$ 
	and $\psi_j(s) \de_t \widetilde{F}^j_1 (y,t)$ are uniformly bounded, too. 
	Using  $\phi_1^j(0)= \phi_2^j(0)$,  $\phi_1^j(\bar t) =\phi_2^j(\bar t)$, and $ \Lip(\phi^j)\leq  \Lip(\phi)$  we obtain:
	\begin{align*}
		& |\psi_j(s)| \leq |\phi_2^j(\lambda^j(s)) - \phi_2^j(0)| + |\phi_1^j(0) - \phi_1^j(s)| \leq \Lip(\phi)(s + \lambda_j(s)), \\
		& |\psi_j(s)| \leq |\phi_2^j(\lambda^j(s)) - \phi_2^j(\bar{t})| + |\phi_1^j(\bar{t}) - \phi_1^j(s)| \leq \Lip(\phi)(2 \bar{t} - (s + \lambda_j(s))) ,
	\end{align*}
	that is
	\begin{equation} \label{eq:psi^j_leq_min}
		|\psi_j(s)| \leq \Lip(\phi) \min \{ s + \lambda_j(s), 2 \bar{t} - (s + \lambda_j(s)) \} .
	\end{equation}
	We claim that
	\begin{equation} \label{eq:min_leq:Delta^j}
		\Delta_j(s) \geq C \min \{ s + \lambda_j(s), 2 \bar{t} - (s + \lambda_j(s)) \} ,
	\end{equation}
	for some $C > 0$ depending only on $\gamma_1$ and $\gamma_2$. Let us assume that $s + \lambda_j(s) \leq 2 \bar{t} - (s + \lambda_j(s))$, i.e., $s+\lambda_j(s) \leq \bar t$. In this case, setting
	\begin{equation}\label{t_1}
		t_1 := \left( 1 + \dfrac{1 - \zeta_j}{1 + \zeta_j} \right)^{-1} \bar{t}  ,
	\end{equation}
	we have $t_1 < \bar{t}$ and  $s \leq t_1$. The latter estimate follows from  \eqref{t_1}, 
	(iv) in \eqref{eq:proplambda} for $\lambda_j^{-1}$, and from:
	\[
	\frac{ \lambda_j(s)}{s} \geq  \Lip(\lambda_j^{-1})\geq \frac{1-\zeta_j}{1+\zeta_j}.
	\]
	By the convexity of $\gamma_1$ and the concavity of $\gamma_2$, we have 
	\begin{align*}
		\Delta_j(s) & = \dfrac{\gamma_2(\lambda_j(s))}{\lambda_j(s)} \lambda_j(s) + \dfrac{|\gamma_1(s)|}{s} s \\
		& \geq \dfrac{\gamma_2(t_1)}{t_1} \lambda_j(s) + \dfrac{|\gamma_1(t_1)|}{t_1} s \\
		& \geq \min \left \{ \dfrac{|\gamma_1(t_1)|}{t_1} , \dfrac{\gamma_2(t_1)}{t_1} \right \} \cdot (s + \lambda_j(s)). 
	\end{align*}
	This is \eqref{eq:min_leq:Delta^j} with
	\[
	C=  \min \left \{ \dfrac{|\gamma_1(t_1)|}{t_1} , \dfrac{\gamma_2(t_1)}{t_1} \right \}.
	\]
	where, notice, $C$ is stable under approximation of $\gamma_1$ and $\gamma_2$.
	The proof of \eqref{eq:min_leq:Delta^j}  in the
	case $s + \lambda_j(s) \geq  \bar{t}  $ is similar.

	By \eqref{eq:psi^j_leq_min}, \eqref{eq:min_leq:Delta^j}, and \eqref{DELTA}, 
	we obtain:
	\[
	\frac{|\psi_j(s)|}{\delta_j(h,s)}\leq  C^{-1}  \Lip(\phi) \frac{\Delta_j(h,s)}{\delta_j(h,s)}\leq  2  C^{-1}   \Lip(\phi).
	\]
	Now, proving a uniform bound for  $\psi_j(s) \de_y \widetilde{F}^j_1 (y,t)$ and $\psi_j(s) \de_t \widetilde{F}^j_1 (y,t)$ is easy
	and, accordingly, our claim \eqref{LIP} follows. This concludes the proof that $u\in \Lip( D)$.

	\medskip
	
	\emph{Step 5.} Since $u$ is Lipschitz continuous in $D$, it admits an extension to $\bar{D}$, that we denote again by $u$.  We are left to show that $u=\phi$ on $\partial D$.

	Let $(y,t) \in \de D$ and  $(h,s) \in \de Q$ be such that  $F(h,s) = (y,t)$. By selecting a sequence $(h_j,s_j) \in Q$ converging to $(h,s)$, we obtain:
	\begin{align} \label{eq:phi_R_bdry}
		u(y,t) & = \lim_{j \to \infty} u(F(h_j,s_j)) \nonumber \\
		& = \lim_{j \to \infty} \rho_1(h_j,s_j) \\
		& = \rho_1(h,s) \nonumber \\
		& = (1 - h) \phi(\gamma_1(s),s) + h \phi(\gamma_2(\lambda(s)),\lambda(s)). \nonumber 
	\end{align}
	Because $(h,s) \in \de Q$, we have   $h \in \{ 0 , 1 \}$ or $s \in \{ 0, \bar{t} \}$. 
	A routine check shows that in any case we have $ (1 - h) \phi(\gamma_1(s),s) + h \phi(\gamma_2(\lambda(s)),\lambda(s))=\phi(y,t)$.
	%
	%
	%
	This concludes the proof.
\end{proof}

\begin{lemma} \label{lem:appendix}
	Let $\Omega_1,\Omega_2 \subset \R^n$ be two open bounded sets and  $F_j,F : \overline{\Omega}_1 \to \overline{\Omega}_2$, $j \geq 1$, be such that:
	\begin{itemize}
		\item[(i)]  $\sup_{j \geq 1} \Lip(F_j,\Omega_1) < \infty$;
		\item[(ii)] $F|_{\Omega_1} $,  $F_{j}|_{\Omega_1}$  are invertible,  $j \geq 1$;
		\item[(iii)]	 $F_j \to F$ pointwise in $\Omega_1$;
		\item[(iv)]
		$F_j(\de \Omega_1) \subset \de \Omega_2$, $F_j(\Omega_1) = \Omega_2$,  $j \geq 1$, and $F(\Omega_1) = \Omega_2$.
	\end{itemize}
	Then $F_j^{-1} \to F^{-1}$ pointwise in $\Omega_2$.
\end{lemma}

\begin{proof}
	Given $y \in \Omega_2$, let $x_j := F_{j}^{-1}(y)$. Up to subsequences, $x_j \to x$, for some $x \in \overline{\Omega}_1$.  
	If  $x \in \de \Omega_1$, then we would have
	\begin{equation} \label{eq:y=F_j(x_j)}
		y = F_j(x_j) = F_{j}(x_j) - F_j(x) + F_j(x),
	\end{equation}
	where, by (iv), $F_j(x) \in \de \Omega_2$, and $|F_{j}(x_j) - F_j(x)| \leq \sup_{j \geq 1} \Lip(F_j,\Omega_1) |x - x_j| \to 0$.  It would follow that 
	$F_j(x)\to y   \in \de \Omega_2$, contradicting   $y \in \Omega_2$. Hence, $x \in \Omega_1$, and  by (iii)     $|F_j(x) - F(x)| \to 0$. Then we have
	\begin{align*}
		|F_j(x_j) - F(x)| & \leq |F_j(x_j) - F_j(x)| + |F_j(x) - F(x)| \\
		& \leq \sup_{j \geq 1} \Lip(F_j,\Omega_1) |x - x_j| +  |F_j(x) - F(x)| \to 0 .
	\end{align*}
	Consequently, $y = F(x)$, i.e., $x = F^{-1}(y)$. This proves that the sequence $F_j^{-1}(y)$ converges to $F^{-1}(y)$ for any $y \in \Omega_2$.
\end{proof}

\begin{remark}
	If $R_{\phi}$ is the left intrinsic graph of  $u\in \Lip(D)$---as in Theorem \ref{thm:sx_graph}--- 
	then $R_{\phi}$ has finite standard area and finite Heisenberg area. This is a straightforward consequence of \eqref{eq:hor_per_X_intr_epi}.
	
	
\end{remark}

\subsection{Right intrinsic graph.}
For small $\zeta$,  the surface $R_{\phi}=\rho(Q)$ is also a right intrinsic graph.
\begin{theorem} \label{lem:dx_graph}
	Let $D = D_{\gamma_1,\gamma_2} \subset \mathbb{W}$ be a lenticular domain and let $\phi:\de D\to\R$ be such that:
	\begin{equation} \label{eq:cond_invert_dx}
		\zeta < \dfrac{\sqrt{721} - 25}{48}.
	\end{equation}
	Then there exist a domain $D^r \subset \mathbb{W}$ and a function $u^r: D^r \to \R$ with the following properties:
	\begin{itemize}
		\item[(i)] $u^r$ represents $R_{\phi}$ as a right intrinsic graph on $D^r$, i.e.,
		\begin{equation} \label{eq:intr_graph_R_dx}
			R_{\phi} = S^r_{u^r} = \big\{ (u^r(y,t),y,t - 2 y  u^r(y,t))\in\mathbb H^1  \, : \, (y,t) \in D^r \big \}  ;
		\end{equation}
		\item[(ii)] $u^r$ is locally Lipschitz continuous in $D^r$.
	\end{itemize}
\end{theorem}

\begin{proof}
	Let $\pi^r:\mathbb H^1 = \R^3 \to\mathbb W =\R^2$ be the projection onto $\mathbb W$ along the integral lines of $X^r$, i.e.,
	\begin{equation} \label{eq:def_proj_pi^r}
		\pi^r(x,y,t) = (y, t+2 xy) , \qquad \text{for all $(x,y,t) \in \He$.}
	\end{equation}
	We define $D^r = \pi^r (R_{\phi})$ and let  $G=(G_1,G_2) : Q\to \R^2$  be the composition $G=\pi^r \circ \rho$, where  $\rho:Q\to \mathbb H^1$ is the parameterization of $R_{\phi}$ introduced  in \eqref{eq:def_rho}:
	\begin{align}
		G_1(h,s) & := \rho_2(h,s) = (1 - h) \gamma_1(s) + h \gamma_2(\lambda(s)) = F_1(h,s)
		\label{eq:G_1} \\
		G_2(h,s) & := \rho_3(h,s) + 2 \rho_1(h,s) \rho_2(h,s)= F_2(h,s) + H_2(h,s), \label{eq:G_2} 
	\end{align}
	where
	\begin{equation*}
		H_2(h,s) := 4 \rho_1(h,s) \rho_2(h,s) = 4 [(1-h) \phi_1(s) + h \phi_2(\lambda(s))][(1 - h) \gamma_1(s) + h \gamma_2(\lambda(s))]  .
	\end{equation*}
	We claim that  $G: Q\to D^r$ is injective, thus proving that $\pi^r: R_{\phi} \to D^r$ is injective.
	In this case we can define $u^r(y,t) = \rho_1( G^{-1} (y,t))$ for $(y,t)\in D^r$.
	
	As in the proof of Theorem \ref{thm:sx_graph},   the injectivity of $G$ follows from  the monotonicity of $s\mapsto \widetilde{\tau}_y(s) := G_2(h_y(s),s) = \tau_y(s) + H_2(h_y(s),s)$, where $\tau_y$ is defined in \eqref{eq:def_tau_y}.
	
	We first estimate the derivative of $s\mapsto H_2(h_y(s),s)$.
	By the definition of $h_y = h_y(s)$ in \eqref{eq:def_h_y},
	the function $s\mapsto \rho_2(h_y, s) = y$ is constant, for fixed $y$, and it  follows that:
	\begin{equation}
		\begin{split}
			[H_2(h_y,s)]'  & = 4 \rho_2(h_y,s) \partial _s \big(\rho_1(h_y,s)\big)
			\\
			&= 4  \rho_2(h_y,s)\big( \partial_ h \rho_1 (h_y,s) h_y'+ \partial_s\rho_1 (h_y,s)  \big).
		\end{split}
	\end{equation}
	Now, by the formula for $h'_y$ in  \eqref{eq:h'/h}, we obtain,
	\[
	\begin{split}
		\big|  \partial_ h \rho_1 (h_y,s) h_y'\big|  & = \Big| \big( \phi_2(\lambda)-\phi_1\big) \frac{(1-h_y) \gamma_1'  +  h_y  \gamma_2'(\lambda)\lambda' }{\gamma_2(\lambda)-\gamma_1}\Big|
		\\
		&
		\leq  \frac{2 \| \phi\|_{\infty}   }{\gamma_2(\lambda)-\gamma_1}  \Lip(\gamma) \max\{1,\Lip(\lambda)\},
	\end{split}
	\]
	and  we notice that
	\[ 
	\frac{|
		\rho_2(h_y,s) |}{\gamma_2(\lambda)-\gamma_1} = \frac {|\gamma_1 + h_y (\gamma_2(\lambda)-\gamma_1) |}{\gamma_2(\lambda)-\gamma_1}   \leq 1 + h_y\leq 2,
	\]
	so that
	\begin{equation}\label{Z1}
		\begin{split}
			\big| 
			\rho_2(h_y,s)  \partial_ h \rho_1 (h_y,s) h_y'\big|  \leq     4  \| \phi\|_{\infty}  \Lip(\gamma) \max\{1,\Lip(\lambda)\}.
		\end{split}
	\end{equation}
	On the other hand, we have: 
	\begin{equation}\label{Z2}
		\begin{split}
			|\rho_2(h_y,s)  \partial_s\rho_1 (h_y,s)|  & = |\rho_2(h_y,s)|  |(1-h_y) \phi_1' + h_y \phi_2' \lambda ' |
			\\
			&
			\leq   \| \gamma \|_{\infty} 
			\Lip(\phi) \max\{1,\Lip(\lambda)\}.
		\end{split}
	\end{equation}
	From \eqref{Z1}, \eqref{Z2}, and  (iv) in \eqref{eq:proplambda} we conclude that:
	\begin{equation}
		\label{Z3}
		| [H_2(h_y,s)]' | \leq 5 \zeta \, \frac{1+\zeta}{1-\zeta}.
	\end{equation}
	By the estimate \eqref{TAU} for $\tau_y'$ and \eqref{Z3}, we obtain:
	\begin{align*}
		\widetilde{\tau}_y'(s) & = \tau_y'(s) + [H_2(h_y(s),s)]' \geq 
		\left[ 1 - \zeta \left( 2 + \dfrac{1 + \zeta}{1 - \zeta} \right) \right] \dfrac{1 - \zeta}{1 + \zeta} -   5 \zeta \, \frac{1+\zeta}{1-\zeta} .
	\end{align*}
	So, by \eqref{eq:cond_invert_dx}, the derivative of $\widetilde{\tau}_y$ is uniformly positive. This demonstrates the injectivity of $G$.
	
	By construction, the function  $u^r := \rho_1 \circ G^{-1}$ realizes (i). Let us now show that $u^r$ is locally Lipschitz continuous in $D_r$. 
	As in Theorem \ref{thm:sx_graph}, we argue by approximation.
	Let  $\phi^j : \de D \to \R$, $j\in\N$, be functions such that $\phi_1^j,\phi_2^j$ are of class $C^1(I)$, 
	satisfy \eqref{eq:approx_phi} and   $\zeta^j \leq \zeta$,   where $\zeta^j$ is as in \eqref{eq:zeta^j}. 
	Without loss of generality, up to a further approximation, we assume that $\gamma_1,\gamma_2\in C^1(I)$.
	
	Let $F^j=(F^j_1, F^j_2) $ be defined as in \eqref{eq:F_1}-\eqref{eq:F_2} 
	and let  $G^j=(G^j_1,G^j_2): Q \to \R^2$ be  defined as in  \eqref{eq:G_1}-\eqref{eq:G_2} 
	with $\phi_1^j,\phi_2^j,\lambda^j$ replacing $\phi^1,\phi^2,\lambda$. We also let 
	$H_2^j := G_2^j - F_2^j$. We have  $F^j, G^j \in C^1(Q)$, for $j \geq 1$. 
	
	Since $G_1^j = F_1^j$ and $G_2^j = F_2^j + H_2^j$,
	the partial derivatives of $G_1^j$, $G_2^j$ can be derived from those of $F_1^j$, $F_2^j$ in \eqref{eq:par_der_F}, and of $H_2^j$:
	\begin{align*}
		& \de_h G_1^j (h,s) = \Delta_j(s) 
		\\
		&   \de_s G_1^j (h,s) = \gamma_1'(s) + h  \Delta_j'(s) \\
		& \de_h G_2^j (h,s) = - 4 \Delta_j(s) \, \rho_1^j(h,s) + \de_h H_2^j (h,s) \\
		& \de_s G_2^j (h,s) = 1 + h(\lambda_{j}'(s) - 1) + 2h(1-h)\cE'(s)  + \de_s H_2^j (h,s) .
	\end{align*}
	As in  \eqref{eq:det_F^j}, the determinant of the Jacobian of $G$ has the form 
	\begin{equation} \label{eq:det_G^j}
		\widetilde{\delta}_j(h,s) = \Delta_j(s) (1 + \widetilde{k}_j(h,s)) ,
	\end{equation}
	where 
	\[
	\widetilde{k}_j(h,s) := k_j(h,s) + \de_s H_2^j (h,s) - \de_s G_1^j (h,s) \de_h H_2^j(h,s)/\Delta_j(s),
	\]
	and $k_j(s)$ is as in \eqref{eq:u^j_det}. Now, recalling that $\zeta_j \leq \zeta$, elementary estimates show that: 
	\begin{equation} \label{eq:est_de_s_H_2}
		|\de_s H_2^j (h,s)|= 4 | \rho_1(h,s) \de_s  \rho_2(h,s) + \rho_2(h,s) \de_s  \rho_1(h,s) | \leq \zeta \, \frac{1+\zeta}{1-\zeta},
	\end{equation}
	and 
	\begin{equation} \label{eq:est_de_s_de_h_/Delta}
		\begin{split}
			\dfrac{|\de_s G_1^j (h,s)||\de_h H_2^j(h,s)|}{|\Delta_j(s)|} & \leq \Lip(\gamma) \, \dfrac{1 + \zeta}{1 - \zeta} \, [8(\phi_2(\lambda) - \phi_1) + 4(\phi_1 + h(\phi_2(\lambda) - \phi_1))] \\
			& \leq 5 \zeta \, \dfrac{1 + \zeta}{1 - \zeta} .
		\end{split}
	\end{equation}
	Now, under \eqref{eq:cond_invert_dx},
	\begin{equation*}
		6 \zeta \, \dfrac{1 + \zeta}{1 - \zeta} < \dfrac{1}{4} ,
	\end{equation*}
	and so we infer that \eqref{eq:est_de_s_H_2} and \eqref{eq:est_de_s_de_h_/Delta} combined with \eqref{eq:est_u^j_det} yield
	\begin{align*}
		|\widetilde{k}_j(h,s)| & < \dfrac{1}{2} + 6 \zeta \, \dfrac{1 + \zeta}{1 - \zeta} < \dfrac{3}{4} .
	\end{align*}
	Hence $\widetilde{\delta}_j(h,s) > \Delta_j(h,s)/4 > 0$ in $Q$,
	and, by the Inversion Function Theorem, $G^j\in C^1(Q;D^r)$ is a diffeomorphism. Moreover, the Jacobian of the inverse function $\widetilde G^j$ is bounded on compact sets of $D_r$ uniformly in $j\in\N$. Thus  $\widetilde G^j$ is locally Lipschitz continuous in $D^r$.
	Since $\widetilde{G}^j \to G^{-1}$ pointwise in $D^r$ by
	Lemma \ref{lem:appendix},    we conclude that $G^{-1}$ and thus $u^r $ are also locally Lipschitz continuous in $D^r$.
\end{proof}



\section{Minimality properties of the horizontally ruled surface $R_\phi$}
\label{CAL}

Let $D = D_{\gamma_1,\gamma_2} \subset \mathbb{W}$ be convex  and $\phi\in\Lip(\de D)$.
We denote by $\nu_{E_{u}}=\omega_1 X +\omega_2 Y  \in \Lip (S_u;H)$  the measure theoretic  inner normal to $S_u= R_\phi$, as in \eqref{NN}.
Let $D^r\subset \mathbb W$ be the open set  constructed in Theorem \ref{eq:cond_invert_dx}.
Using an idea introduced  by Barone Adesi, Serra Cassano and Vittone in  \cite{BaroneAdesi_SerraC_Vitt_2007}, in this section we prove that the left graph $S_u$ can be calibrated in the right cylinder $\mathbb V\cdot D^r$.

We recall that by $E_u$, $E^r_{u^r}$ we denote respectively the left intrinsic epigraph of $u : D \to \R$ and the right intrinsic epigraph of $u^r : D^r \to \R$.

\begin{lemma} \label{lem:calibration}
	Assume that $\zeta$ satisfies   \eqref{eq:cond_invert_dx}. Then there exists a locally Lipschitz continuous horizontal section 
	$\omega =\omega_1 X + \omega_2 Y  \in\Lip_{loc} (\mathbb V\cdot D^r;H)$ such that:
	\begin{itemize}
		\item[(i)] $\omega_1 ^2 +\omega_2^2 =1$ in $ \mathbb V\cdot D^r$;
		\item[(ii)] $ \div \omega =0$,  $\cL$-a.e.~in $ \mathbb V\cdot D^r$;
		\item[(iii)] $\omega = \nu _{E_{u}}$ on  $ S_u$.  
	\end{itemize}
	
\end{lemma}

\begin{proof} The surface $R_\phi  = \de E_{u} \cap  D\cdot\mathbb V$ is parametrized by
	\begin{equation*}
		\rho(h,s) = p_1(s) + h( p_2(\lambda(s)) - p_1(s)),  \qquad \text{$(h,s) \in Q = (0,1) \times I$,}
	\end{equation*}
	and, by (ii) of Theorem  \ref{lem:exlambda},    $p_2(\lambda(s)) - p_1(s)\in H_{\rho(h,s) }$ for any $h\in (0,1)$, and thus $0\neq p_2(\lambda(s)) - p_1(s)= \bar \alpha(s) X (\rho(h,s)) +\bar \beta(s) Y(\rho(h,s))$, for functions $\bar\alpha,\bar\beta : I \to\R$
	independent of $h$. We define $\alpha,\beta: I\to\R$ 
	\[
	\alpha = \frac{\bar \alpha }{\sqrt{\bar\alpha^2+\bar\beta^2}}\qquad \textrm{and}\qquad
	\beta  = \frac{\bar \beta  }{\sqrt{\bar\alpha^2+\bar\beta^2}}.
	\]
	Since $R_\phi = S_u$ is an intrinsic graph along $X$, we have $\beta(s)\neq  0$ for any $s$, and without loss of generality
	we can assume that $\beta>0$. The inner normal $\nu_{E_u}$ is therefore given by the formula
	\begin{equation}   
		\nu_{E_u}(p) = \beta(s) X(p) - \alpha(s) Y(p), \qquad \text{for $p = \rho(h,s) \in  R_\phi$.} 
	\end{equation}

	\emph{Step 1.} We construct $\omega$ satisfying (i) and (iii). 
	By \eqref{eq:cond_invert_dx} and Theorem  \ref{lem:dx_graph},  $S_u=R_\phi= S^r_{u^r}$ 
	is the right intrinsic graph of a function $u^r\in\Lip_{loc}( D^r)$,
	$\sigma(y,t)  
	= (u^r(y,t), y,  t - 2  y u^r(y,t))$ for  $(y,t) \in D^r$,	and we can define the functions $\bar \lambda,\bar \mu : D^r\to\R$ via the identity
	\begin{equation*}
		\nu_{E_u} (\sigma ) = \bar  \lambda  X(\sigma)+ \bar \mu Y(\sigma).
	\end{equation*}
	Finally, we define $\lambda,\mu \in \Lip_{loc}(\mathbb V\cdot D^r)$ letting $\lambda (p) = \bar  \lambda(\pi^r(p))$ and $\mu(p)=  \bar \mu(\pi^r(p))$,
	where  $\pi^r:\He\to\mathbb W$ is  the projection onto $\mathbb{W}$ along the integral lines of  $X^r$. 
	We can finally define 
	$\omega\in\Lip_{loc}(\mathbb V\cdot D^r ; H)$ letting, for $p = (x,y,t) \in \mathbb V \cdot D^r$, 
	\begin{align*}
		\omega (p)  :&= \lambda(p) X(p) + \mu(p) Y(p) \\
		& =\bar  \lambda(y,t+2yx) X(p) + \bar \mu(y,t+2yx) Y(p).
	\end{align*}
	\noindent 
	By construction, we have $\omega= \nu_{E_u}$ on $R_\phi=S_u$.
	The vector field  $\omega$ is locally Lipschitz continuous in $\mathbb V \cdot D^r$, because so are $\lambda$ and $\mu$.
	

	\medskip
	
	%
	%

	\emph{Step 2.} Let $\Phi_s:\He\to\He$, $s\in\R$, be the flow of $X^r$.
	We claim that, for a.e.~$p=(x,y,t)  \in \mathbb V \cdot{D}^r$,
	\begin{equation} \label{eq:inv_div_H_nu}
		\div \omega  (\Phi^r_s(p)) = \div  \omega (p) , \qquad \text{for any $s \in \R$.}
	\end{equation}
	Indeed, using the coordinates $(\eta,\tau) =   \pi^r(p) = (y,  t+2yx)$, we have: 
	\begin{equation}
		\begin{split}
			\div  \omega  (p) & = X(p)(\lambda(p)) + Y(p)(\mu(p)) \\
			& = \dfrac{\de}{\de x}[\bar \lambda(y,t+2yx)] + 2y  \dfrac{\de}{\de t}[\bar \lambda(y,t +2yx)] +
			\\
			&\qquad + \dfrac{\de}{\de y}[\bar \mu(y,t +2yx)] - 2x \dfrac{\de}{\de t}[ \bar  \mu(y,t +2yx)] 
			\\
			& = \dfrac{\de\bar  \mu}{\de \eta}  (y, t+2yx) + 4 y \, \dfrac{\de \bar \lambda}{\de \tau}  (y, t+2yx) \\
			& = \dfrac{\de\bar  \mu}{\de \eta}  (\eta,\tau) + 4 \eta \, \dfrac{\de \bar \lambda}{\de \tau}  (\eta,\tau ) .
		\end{split}
	\end{equation}
	This   gives \eqref{eq:inv_div_H_nu}, because $\pi^r(\Phi^r_s(p)) = \pi^r(p)$ for any $s\in\R$.
	\medskip
	
	\emph{Step 3.} We claim that 
	\begin{equation} \label{eq:div_H_nu=0_on_R}
		\div \omega (p)  = 0 , \qquad \text{for ${\mathscr H}^2$-a.e.~$p\in R_\phi$.}
	\end{equation}
	Along with Step 2, this will conclude the proof of (ii).

	Let $Z$ be the horizontal vector field in  $\mathbb V \cdot D^r$ defined by
	\begin{equation*}
		Z  := \mu X - \lambda Y.
	\end{equation*}
	
	From (ii) in Theorem \ref{lem:exlambda}, the functions $\alpha$ and $\beta$ and thus also the functions $\lambda$ and $\mu$ are constant along the integral lines of $Z$ foliating $R_\phi$:
	\begin{align}
		& 0 = Z \lambda= \mu X\lambda- \lambda Y\lambda, \qquad \text{on $R_\phi $} \label{eq:tau(lambda)_1} \\
		& 0 = Z \mu = \mu X\mu - \lambda Y\mu, \qquad \text{on $R_\phi$.} \label{eq:tau(lambda)_2}
	\end{align}
	On the other hand,  from $\lambda^2 + \mu^2 = 1$ we deduce that:
	\begin{align}
		& 0 = \dfrac{1}{2}  X(\lambda^2 + \mu^2) = \lambda X\lambda + \mu X\mu ,
		\qquad \text{a.e.~in $\mathbb V  \cdot D^r$,} \label{eq:X(lambda^2+mu^2)} \\
		& 0 = \dfrac{1}{2} \, Y(\lambda^2 + \mu^2) = \lambda Y\lambda+ \mu Y\mu,  \qquad \text{a.e.~in $\mathbb V \cdot{D}^r$.} \label{eq:Y(lambda^2+mu^2)}
	\end{align}
	Recalling that $\lambda^2+\mu^2 \equiv 1$, a direct computation yields
	\begin{equation*}
		\begin{split}
			\div \omega & = X \lambda + Y \mu \\ 
			& = \mu(\mu X \lambda - \lambda Y \lambda) + \lambda(\lambda X \lambda + \mu X \mu) - \lambda(\mu X \mu - \lambda Y \mu) + \mu(\lambda Y \lambda + \mu Y \mu) \, .
		\end{split}
	\end{equation*}
	Therefore, applying \eqref{eq:tau(lambda)_1}, \eqref{eq:tau(lambda)_2}, \eqref{eq:X(lambda^2+mu^2)}, and \eqref{eq:Y(lambda^2+mu^2)}, we obtain
	\begin{equation*}
		\div \omega = 0 , \qquad \text{${\mathscr H }^2$-a.e.~on $R_\phi$.} 
	\end{equation*}
\end{proof}

\begin{remark} \label{rem:u_in_C11H}
	Note that, by \emph{(iii)} of Lemma \ref{lem:calibration} above, $\nu_{E_u}$ is Lipschitz continuous on the intrinsic graph $S_u$. This implies that $u \in C^{1,1}_H(D)$.
\end{remark}

By Lemma \ref{lem:calibration}, $E_{u^r}$ is a minimizer of the $H$-perimeter in the right cylinder $ \mathbb V \cdot D^r$:

\begin{theorem} \label{thm:E_minimizer_P_H} 
	Let $\phi\in \Lip(\partial D)$ satisfy  \eqref{eq:cond_invert_dx} and let $u^r$ be as in Theorem  \ref{lem:dx_graph}.
	Then for any $\cL$-measurable set $F\subset  \mathbb V \cdot D^r$ such that $E_{u^r}^r \difsim F \Subset  \mathbb V \cdot D^r$
	we have: 
	\begin{equation} \label{eq:E^r_per_min}
		\p_H(E^r_{u^r} ;  \mathbb V \cdot D^r) 
		\leq \p_H(F; \mathbb V \cdot D^r).
	\end{equation}
\end{theorem}
\begin{proof} 
	It follows from Lemma  \ref{lem:calibration} 
	and Theorem 2.1 in \cite{BaroneAdesi_SerraC_Vitt_2007}.
\end{proof}

We have the set-theoretical identities:
\[
\partial E_u \cap D\cdot \mathbb V = S_u = R_\phi = S^r_{u^r} = \partial E_{u^r}^r \cap \mathbb V \cdot D^r,
\]
where $u^r \in \Lip_{loc}(D_r)$ represents $R_\phi$ as an intrinsic graph along $X^r$.
Let  $n_R$ be the Euclidean unit normal to $R_{\phi}$ and define $ N_E := \left( \langle n_{R} , X \rangle, \langle n_{R} , Y \rangle \right) \in \R^2$,
where $\langle\cdot,\cdot \rangle $ is the standard scalar product in $\R^3$. Then we have the following identities:
\begin{equation} 
	\label{eq:formula_per_H_Lip}
	\p_H(E_u; D\cdot \mathbb V)  = \int_{S_u } |N_{E_u}| \, d \mathscr H^{2}= \int_{S_{u^r}^r } |N_{E_u}| \, d \mathscr H^{2}= 	\p_H(E^r_{u^r} ;  \mathbb V \cdot D^r) .
\end{equation}
Thus, if  for some competitor $v\in \Lip(D)$ with $v=\phi$ on $\de D$ we have $S_v = \de F \cap  \mathbb V \cdot D^r$  for some $F\subset  \mathbb V \cdot D^r$,
then inequality 
\eqref{eq:E^r_per_min} reads:
\begin{equation} \label{eq:HArea[u]leqHArea[v]}
	\cA_D(u)\leq \cA_D(v) .
\end{equation}

In the next Lemma we provide a condition guaranteeing that  the left graph $S_v$ of a competitor $v$ is also a right graph over the domain $D^r$.

Let us fix   some notation. Given $y \in \R$, we denote by $D_y$ the set
\begin{equation*}
	D_y := \{ t \in I \, : \, (y,t) \in D \} = (t_y^-,t_y^+) ,
\end{equation*}
and we define the interval $Y_D := \{ y \in \R \, : \, D_y \neq \emptyset \}$.
Given a continuous function $v : D \to \R$, we  let
\begin{equation} \label{eq:def_beta}
	\beta^v_y(t): = t + 4yv(y,t) , \qquad t \in D_y .
\end{equation}

\begin{theorem} \label{lem:minimization_intr_gr} Let $\phi \in \Lip(\de D)$ satisfy  \eqref{eq:cond_invert_dx}
	and let $v\in C(\bar D)$ be such that $v=\phi$ on $\de D$ and $\cB v\in L^{\infty}_{loc}(D)$.
	Assume that for every $y \in Y_D$ we have:
	\begin{equation} \label{eq:cond_beta_incr}
		\text{$D_y \ni t \mapsto \beta^v_y(t)$ is increasing.}
	\end{equation}
	Then  $\cA_D(u) \leq \cA_D(v)$ with equality if and only if  $u=v$.
\end{theorem}

\begin{proof} Let   $S_v$  be the intrinsic graph of $v$. Using \eqref{eq:cond_beta_incr} and the fact that $v=u$ on $\de D$ we obtain:
	\[
	\begin{split}
		\pi^r(S_v) & = \big\{(y,\beta_y^v(t)) \in\mathbb W : (y,t)\in D\big\}
		\\
		& = \bigcup_{y \in Y_D} \{y\} \times \big( \beta_y^v (t_y^-), \beta_y^v (t_y^+)\big)
		\\
		& = \bigcup_{y \in Y_D} \{y\} \times \big( \beta_y^u  (t_y^-), \beta_y^u (t_y^+)\big) 
		\\
		&= \pi^r (S_u) = D^r.
	\end{split}
	\]
	Since the map $D\ni (y,t) \to (y, \beta_y^v(t))$ is injective, there exists a function $v^r\in C(D^r)$  
	such that $S^r_{v^r}= S_v$  and $v^r= u^r$ on $\de D^r$.
	
	Now, by applying the Divergence Theorem proved in \cite{FSSC_Rect_Per_M_Ann_2001}, we have:
	\[
	0 = \int_{E^r_{u^r} \difsim E^r_{v^r}} \div \omega\,  d\cL  =\int _{S_v}\langle \omega, \nu_{E_v} \rangle
	d\mu_{E_v}  - \int _{S_u}\langle \omega, \nu_{E_u} \rangle d\mu_{E_u} .
	\]
	Using $\langle \omega, \nu_{E_u} \rangle=1$ on $S_u$
	and $ \langle \omega, \nu_{E_v} \rangle\leq 1$, and the area formula \eqref{eq:hor_per_X_intr_epi}
	we conclude that  $\cA_D(u) \leq \cA_D(v)$.
	
	In the case of equality  $\cA_D(u) = \cA_D(v)$, we deduce that $\nu_{E_v} = \omega $, $\mu_{E_v}$-a.e.~on $S_v$. Namely, at the point 
	$p =\big(v(y,t),y,t+2y v(v,t)\big)\in S_v$we have
	\begin{equation}\label{M1}
		\frac{1}{\sqrt{1+\cB v(y,t)^2}} X(p) - 
		\frac{\cB v(y,t)}{\sqrt{1+\cB v(y,t)^2}}  Y(p) = \nu_{E_v} (p) = \omega (p).
	\end{equation}
	Theorem 1.2 in \cite{Monti_Vittone_2012} implies that $v \in C^1_H(D)$, and in particular it is locally intrinsic Lipschitz, i.e., $\cB v \in L^\infty_{loc}(D)$.  Let $\kappa:(-\delta,\delta) \to  D$, for some $\delta>0$ be an integral curve of the vector field
	$Z= \partial _y-4v(y,t) \partial _t$, $\dot \kappa = Z(\kappa)=(1,-4v(\kappa))$. By Theorem 1.2 in \cite{Bigolin_Carav_SC_2015}, 
	the function $s\mapsto \vartheta(s) : = v(\kappa(s))$ is differentiable for a.e.~$s$ and 
	\begin{equation}
		\label{theta}
		\dot\vartheta(s) = \cB v (\kappa(s)),\quad s\in (-\delta,\delta),
	\end{equation}
	and thus, in fact, $\vartheta \in C^1(-\delta,\delta)$.  Let $\gamma\in C^1\big( (-\delta,\delta);S_v)\big)$ be the lift of $\kappa$ on the graph of $v$, namely,
	\[
	\gamma(s) = \big( \vartheta(s), \kappa_1(s), \kappa_2(s) + 2 \kappa_1(s) \vartheta(s)\big).
	\]
	The derivative of $\gamma$ is
	\begin{equation}\label{M2}
		\begin{split}
			\dot \gamma (s) &=\dot\vartheta(s)  \big(1,0,2\kappa_1(s) \big) + \big( 0, 1,- 2\vartheta(s)\big) 
			\\
			& = \cB v (\kappa(s)) X (\gamma(s) )+ Y(\gamma(s)). 
		\end{split}
	\end{equation}
	Letting
	$ J  \omega = J(\omega_1 X + \omega_2 Y) = - \omega_2 X + \omega_1 Y$ and denoting by $\bar\gamma$ the arc length parameterization of $\gamma$, from \eqref{M1} and \eqref{M2} we deduce that
	\[
	\dot{\bar\gamma} = J\omega (\bar\gamma).
	\]
	A direct check shows that the integral curves of $J\omega$ are horizontal lines. This follows 
	from  the fact the $S_u$ is foliated by horizontal lines and from the construction of $\omega$.
	
	The previous argument proves that   the graph  $S_v$ is foliated by horizontal lines and since $v=u$ on $\partial D$ we conclude that $v=u$ also in $D$.
\end{proof}

\begin{remark} \label{rem:u_stationary}
	We observe that $u$ is stationary for $\cA_D$. To prove this, we apply Theorem \ref{lem:minimization_intr_gr}. First, note that the transformation $D \ni (y,t) \mapsto A(y,t) := (y,\beta_y^u(t))$, where $\beta_y^u$ is given by \eqref{eq:def_beta}, coincides with $G \circ F^{-1}$, where $F$ and $G$ are defined in \eqref{eq:F_1}, \eqref{eq:F_2} and \eqref{eq:G_1}, \eqref{eq:G_2}, respectively. Hence, $A^{-1} = F \circ G^{-1}$ and, by the properties of $F$ and $G$, $A^{-1}$ is locally Lipschitz continuous in $D^r$. In particular, given $K \Subset D^r$, there exists a constant $C_K > 0$ such that
	\begin{equation*}
		|\de_{\tau} A^{-1}(\eta,\tau)| \leq C_K , \qquad \text{for all $(\eta,\tau) \in K$,}
	\end{equation*}
	and consequently, by the definition of $A$,
	\begin{equation} \label{eq:low_bd_beta'}
		\dfrac{d}{dt} \beta^u_y(t) \geq C_K^{-1} , \qquad \text{for all $(y,t) \in A^{-1}(K)$.}
	\end{equation}
	Let $v \in C^{\infty}_c(D)$. When $(y,t) \notin \spt(v)$, we have
	\begin{equation*}
		\dfrac{d}{dt} \beta^{u + \epsilon v}_y(t) = \dfrac{d}{dt} \beta^u_y(t) > 0 .
	\end{equation*}
	If $(y,t) \in \spt(v)$, then by \eqref{eq:low_bd_beta'} it follows that
	\begin{equation*}
		\dfrac{d}{dt} \beta^{u + \epsilon v}_y(t) = \dfrac{d}{dt} \beta^{u}_y(t) + \epsilon \de_t v (y,t) \geq C_{A(\spt(v))}^{-1} + \epsilon \de_t v (y,t) .
	\end{equation*}
	Provided that $\epsilon > 0$ is small enough, we then infer that, for all $y \in Y_D$, the derivative of $\beta^{u + \epsilon v}_y$ is positive on $D_y$. This allows us to apply Lemma~\ref{lem:minimization_intr_gr} to $u + \epsilon v$, from which we deduce that $\cA_D(u) \leq \cA_D(u + \epsilon v)$, and in particular,
	\begin{equation*}
		\dfrac{d}{d \epsilon} \Big|_{\epsilon = 0} \cA_D(u + \epsilon v) = 0 .
	\end{equation*}
	Hence, $u$ is a stationary point for $\cA_D$, as desired.
\end{remark}

\section{A regularity result for $H$-perimeter minimizers}
\label{REG}

We use the results obtained in the previous sections  to prove a regularity theorem for minimizers of the $H$-area.

For $0<\rho<r$, let  $\gamma_{r,\rho}: I\to\R$, $I:= (0,2r)$, be the function
\begin{equation*}
	\gamma_{r,\rho}(s) := \dfrac{\rho}{r^2} \, s (2r - s) \qquad \text{for $s \in I$,}
\end{equation*}
and consider the domain $D_{r,\rho} =(0,- r) + D_{\gamma_1,\gamma_2}$, where $D_{\gamma_1,\gamma_2}$ is as in \eqref{D} with $\gamma_1 = - \gamma_{r,\rho}$, $\gamma_2 = \gamma_{r,\rho}$, i.e.,  
\begin{equation*}
	D_{r,\rho} := \{ (y,t-r) \in \mathbb{W} \, : \, t \in I, \, |y| < \gamma_{r,\rho}(t) \} .
\end{equation*}
Moreover, for $w_0 = (y_0,t_0)$, we set
\begin{equation*}
	D_{r,\rho}(w_0) := w_0 + D_{r,\rho} \qquad \text{and} \qquad 
	B_r (w_0) = \big\{ w\in\mathbb W\,:\, |w-w_0|< r\big\} ,
\end{equation*}
where $|\cdot|$ is the standard norm on the plane. 
Notice that, for $\rho<r$, $D_{r,\rho}(w_0) \subset  B_{r}(w_0)$.

Finally,  consider the  half-balls:
\begin{equation*} 
	B_r^{+}(w_0) := \{ (y,t) \in B_r(w_0) \, : \, y > y_0 \} \quad \text{and} \quad B_r^{-}(w_0) := \{ (y,t) \in B_r(w_0) \, : \, y < y_0 \} ,
\end{equation*}
and, for a function $v: B_{r}(w_0)\to \R $, define $\ell^v_{r} : (0,r)\to [0,\infty]$ by
\begin{equation}\label{elle}
	\ell^v_{r}  (\rho;w_0)=  \max\{ \Lip(v, B_{r}^+(w_0) \setminus D_{r,\rho}(w_0)), \Lip(v, B_{r}^-(w_0) \setminus D_{r,\rho}(w_0)) \} .
\end{equation}

\begin{theorem} \label{thm:reg}
	Let $\Omega \subset \mathbb{W}$ be an open set and let  $v \in C(\Omega)$ be a function such that $\cB v \in L^\infty_{loc}(\Omega)$. 
	Assume that:
	\begin{itemize}
		\item[(i)] The function $v$ is a minimizer of the $H$-area for compact perturbations. Namely, for any 	 $\widetilde v\in C(\Omega)$ such that $ \spt(v-\widetilde v ) \Subset \Omega' \Subset \Omega $ we have:
		\begin{equation} \label{eq:min_H_Area}
			\cA_{\Omega'}(v)\leq \cA_{\Omega'}(\widetilde v) .
		\end{equation}
		
		\item[(ii)] For all $w_0 = (y_0,t_0) \in \Omega$  there exists $r > 0$ such that $B_{r}(w_0) \subset \Omega$ and 
		\begin{equation} \label{eq:cond_Lip_loc}
			\lim_{\rho\to 0^+} \rho 	\,
			\ell^v_{r}  (\rho;w_0)=0 . 
		\end{equation}
	\end{itemize}
	Then $v\in\Lip_{loc}(\Omega) \cap C^{1,1}_{H,loc}(\Omega)$ and its left intrinsic graph $S_v$ is foliated by horizontal straight lines.
\end{theorem}

\begin{proof} 	We claim that  for every $w_0 \in \Omega$ there exists a neighborhood  $ U\subset \mathbb W$ of $w_0$ such that the  
	$S_v \cap U\cdot\mathbb V $
	is foliated by horizontal straight line segments. 
	
	Up to a translation, we can without loss of generality assume  that $w_0 = (0,r)$, where $r > 0$ is a radius for which \eqref{eq:cond_Lip_loc} holds.
	We estimate  the parameter $\zeta$ in \eqref{zeta}
	for the   domain $D_{r,\rho}(w_0) = D_{-\gamma_{r,\rho},\gamma_{r,\rho}}$ and boundary datum $\phi = v|_{\de D_{r,\rho}(w_0) }$:
	\begin{align*}
		\zeta & = 4 \big (\| \gamma_{r,\rho} \|_{\infty} + \Lip(\gamma_{r,\rho}) \big )\big  (\| \phi \|_{\infty} + \Lip(\phi\circ(s,\gamma_{r,\rho}(s))\big ) \\
		& \leq 4 \left( 1 + 2r \right) \Lip(\gamma_{r,\rho}) \big( \| v \|_{L^{\infty}(B_{r}(w_0))} + \ell^v_{r} (\rho;w_0) (1 + \Lip(\gamma_{r,\rho})) \big) . 
	\end{align*}
	We used \eqref{elle} to get the inequalities:
	\[
	\Lip(\phi\circ(s,\gamma_{r,\rho}(s))\leq 	\Lip(\phi) (1 + \Lip(\gamma_{r,\rho})) \leq \ell^v_{r} (\rho;w_0) (1 + \Lip(\gamma_{r,\rho})).
	\]
	Using now $ \Lip(\gamma_{r,\rho})= 2\rho/r$ we obtain the final estimate:
	\[
	\zeta\leq \dfrac{8}{r} \left( 1 + 2r \right) \left[ \rho  \| v \|_{L^{\infty}(B_{r}(w_0))} + \left( 1 + 2 \dfrac{\rho}{r} \right) \rho \, \ell^v_{r}(\rho;w_0) \right].
	\]
	
	By \eqref{eq:cond_Lip_loc},   for  $\rho$ sufficiently small, $\zeta$ satisfies  condition \eqref{eq:cond_invert_dx}.
	By Theorems \ref{lem:exlambda} and \ref{thm:sx_graph}, and by Remark \ref{rem:u_in_C11H}, the ruled surface $R_\phi$ is the left graph $S_u$  of a function $u\in \Lip (D_{r,\rho}(w_0)) \cap C^{1,1}_H(D_{r,\rho}(w_0))$.
	Moreover, by Theorem \ref{lem:dx_graph} $S_u$ is also a right graph and by the calibration argument of 
	Section \ref{CAL} it is a minimizer for the area with boundary datum $\phi$ in a suitable class of competitors.
	
	The function $v$ belongs to this class of competitors if the function $\beta^v_y$ satisfies condition \ref{eq:cond_beta_incr}. We check this condition. Consider $y \in Y_{D_{r,\rho}(w_0)}=(-\rho,\rho)$, $y\neq0$. Notice that if $(y,t)\in D_{r,\rho}(w_0)$ then 
	$(y,t) \in D_{r,\rho}(w_0) \setminus D_{r,|y|/2}(w_0)$.
	Thus, by \eqref{elle} we have:
	\begin{align*} 
		\dfrac{d}{dt} \beta^v_y(t) = 1 + 4 y  \de_t v (y,t) \geq 1 - 4 |y| \ell^v_{r} (|y|/2;w_0) .
	\end{align*}
	Because $|y| < \rho$,  selecting $\rho > 0$ small enough, by \eqref{eq:cond_Lip_loc}, we obtain
	\begin{equation*}
		\dfrac{d}{dt} \beta^v_y(t) \geq \dfrac{1}{2} > 0 ,
	\end{equation*}
	thus proving  the monotonicity of $ t \mapsto \beta^v_y(t)$, for every $y \in Y_D$.

	By Lemma \ref{lem:minimization_intr_gr}, the function $v$ is an admissible competitor for $u$ and therefore
	$$\cA_{D_{r,\rho}(w_0)}(u) \leq \cA_{D_{r,\rho}(w_0)}(v).$$ On the other hand, by assumption (i) we also have $\cA_{D_{r,\rho}(w_0)}(v) \leq \cA_{D_{r,\rho}(w_0)}(u)$, and thus $\cA_{D_{r,\rho}(w_0)}(v) = \cA_{D_{r,\rho}(w_0)}(u)$. By the equality characterization in
	Lemma \ref{lem:minimization_intr_gr}, we infer that $v=u$, and we conclude. 
\end{proof}


\end{document}